\documentclass{amsart} 
\usepackage{amssymb, amsmath} 
\usepackage{enumerate}
\usepackage{xypic}
\input xy
\xyoption{all}

\newtheorem{crl}{Corollary}[section]

\newtheorem{lmm}{Lemma}[section]

\newtheorem{prp}{Proposition}[section]

\newtheorem{thm}{Theorem}[section]

\theoremstyle{definition}
\newtheorem{dfn}{Definition}[section]
 
\theoremstyle{remark}
\newtheorem{rem}{Remark}[section] 

\newcommand{\R}{\mathbb{R}}
\newcommand{\Z}{\mathbb{Z}}
\newcommand{\Q}{\mathbb{Q}}
\newcommand{\C}{\mathbb{C}}

\renewcommand{\P}{\mathbb{P}}

\newcommand{\mmP}{\mathcal{P}}

\newcommand{\mmU}{\mathcal{U}}

\newcommand{\mmK}{\mathcal{K}}

\newcommand{\mmW}{\mathcal{W}}
\newcommand{\mmD}{\mathcal{D}}

\newcommand{\wC}{\widetilde{C}}

\newcommand{\wZ}{\widetilde{\Z}}

\newcommand{\whP}{\widehat{\mmP}}

\newcommand{\hch}{\widehat{\ch}}
\newcommand{\och}{\overline{\ch}}

\newcommand{\whK}{\widehat{K}}
\newcommand{\wmD}{\widetilde{\mmD}}

\newcommand{\bj}{\boldsymbol{j}}

\newenvironment{enumerate*}[1][{}]{\begin{itemize}\renewcommand{\labelitemi}{#1}}{\end{itemize}}
\DeclareMathOperator*{\ch}{ch}

\DeclareMathOperator*{\Cub}{Cub}

 \DeclareMathOperator*{\Hom}{Hom} \DeclareMathOperator*{\im}{im}

\DeclareMathOperator*{\Sp}{Sp}
 
\DeclareMathOperator*{\tr}{tr} 
 
\DeclareMathOperator*{\HoFib}{HoFib}
 \newcommand{\Spn}{\Sp\nolimits}\DeclareMathOperator*{\Id}{Id}
\newcommand{\vist}{\begin{flushright}
$\square$
\end{flushright}}

\title{Adams operations on higher arithmetic K-theory}
\author{Elisenda \textsc{Feliu}}
\email{efeliu@ub.edu}
\address{Universitat de Barcelona \\ Gran Via de les Corts Catalanes, 585 \\ 08007 Barcelona, Spain}
\thanks{Supported partially by the DGICYT BFM2003-02914}

\begin{document}

\maketitle

\begin{abstract} We construct Adams operations on the rational higher arithmetic K-groups of a proper arithmetic variety. The definition applies to the higher arithmetic K-groups given by Takeda as well as to the groups suggested by Deligne and Soul\'e, by means of the homotopy groups of the homotopy fiber of the regulator map. They are compatible with the Adams operations on algebraic $K$-theory. The definition relies on the chain morphism representing Adams operations in higher algebraic $K$-theory given previously by the author. In this paper it is shown that a slight modification of this chain morphism commutes strictly with the representative of the Beilinson regulator given by Burgos and Wang.
\end{abstract}

\section*{Introduction}

This paper contributes to the development of a \emph{higher arithmetic intersection theory} following the steps of the higher algebraic intersection theory but suitable for arithmetic varieties. In \cite{BFChow}, the author, together with Burgos, defined the higher arithmetic Chow ring for any arithmetic variety over a field, extending the construction given by Goncharov in \cite{Goncharov} which was valid only for proper arithmetic varieties. The question that arises is whether these groups are related to the higher arithmetic $K$-groups as given by Takeda or as suggested by Deligne and Soul\'e (see below). To this end, and inspired by the algebraic analogue, in this paper we endow the higher arithmetic $K$-groups of an arithmetic variety (tensored by $\Q$) with a (pre)-$\lambda$-ring structure.

Let $X$ be an arithmetic variety over the ring of integers $\Z$.
In order to define the arithmetic Chern character on hermitian
vector bundles, Gillet and Soul{\'e} have introduced in
\cite{GilletSouleClassesII} the arithmetic $K_0$-group, denoted by
$\widehat{K}_0(X)$. They endowed $\widehat{K}_0(X)$ with a
pre-$\lambda$-ring structure, which was shown to be a
$\lambda$-ring structure by R{\"o}ssler in \cite{Roessler1}. This
group fits in an exact sequence
\begin{equation*}\label{introkseq2}
K_1(X) \xrightarrow{\rho} \bigoplus_{p\geq 0} \mmD^{2p-1}(X,p)/\im d_{\mmD}
\rightarrow \whK_0(X) \rightarrow K_0(X) \rightarrow 0, \tag{*}
\end{equation*}
 with $\rho$ the Beilinson regulator (up to a constant factor) and $\mmD^{*}(X,p)$ the Deligne complex of differential forms with $p$-twist computing
 Deligne-Beilinson cohomology with $\R$ coefficients and twisted by $p$, $H_{\mmD}^{*}(X,\R(p))$.

Two different definitions for higher arithmetic K-theory
have been proposed. Initially, it was suggested
by Deligne and Soul{\'e} (see \cite[$\S$III.2.3.4]{Soule2}  and
\cite[Remark 5.4]{Delignedeterminant}) that these groups should
fit in a long exact sequence
$$\dots \rightarrow K_{n+1}(X) \xrightarrow{\rho} H^{2p-n-1}_{\mmD}(X,\R(p)) \rightarrow \widehat{K}_n(X) \rightarrow
K_n(X) \rightarrow \dots, $$ extending the exact sequence
\eqref{introkseq2}, with  $\rho$  the Beilinson
regulator. This can be achieved by defining $\widehat{K}_n(X)$ to
be the homotopy groups of the homotopy fiber of a representative
of the Beilinson regulator (for instance, the representative
``$\ch$'' defined by Burgos and Wang in \cite{Burgos1}).

If $X$ is proper,  Takeda has given in \cite{Takeda} an
alternative definition of the \emph{higher arithmetic $K$-groups}
of $X$, by means of \emph{homotopy groups modified} by the representative
of the Beilinson regulator ``$\ch$''.  We denote these higher
arithmetic $K$-groups by $\widehat{K}_n^{T}(X)$. In this case,
these groups fit in exact sequences
$$
K_{n+1}(X) \xrightarrow{\rho} \bigoplus_{p\geq 0} \mmD^{2p-n-1}(X,p)/\im d_{\mmD}
\rightarrow \whK_n^T(X) \rightarrow K_n(X) \rightarrow 0,
$$
analogous to \eqref{introkseq2}. The two definitions do not agree, but, as proved by Takeda, they
are related by a natural isomorphism:
$$ \widehat{K}_n(X) \cong \ker (\ch: \widehat{K}_n^{T}(X) \rightarrow \mmD^{2p-n}(X,p)),\qquad n\geq 0.$$

In this paper we give a pre-$\lambda$-ring structure on the higher
arithmetic $K$-groups $\widehat{K}_n(X)_{\Q}$ and
$\widehat{K}_n^T(X)_{\Q}$. It is compatible with the
$\lambda$-ring structure on the algebraic $K$-groups, $K_n(X)$,
defined by Gillet and Soul{\'e} in \cite{GilletSouleFiltrations}.
Moreover, for
$n=0$ we recover the $\lambda$-ring structure of
$\widehat{K}_0(X)$.

 More concretely, we construct \emph{Adams operations}
 \begin{eqnarray*}
\Psi^k: \widehat{K}_n(X)\otimes \Q  & \rightarrow & \widehat{K}_n(X)\otimes \Q,\qquad k\geq 0, \\
\Psi^k: \widehat{K}_n^T(X)\otimes \Q  & \rightarrow & \widehat{K}_n^T(X)\otimes \Q,\qquad k\geq 0,
\end{eqnarray*}
which, since we have tensored by $\Q$, induce $\lambda$-operations on
$\widehat{K}_n(X)\otimes \Q$ and  $\widehat{K}_n^T(X)\otimes \Q$.

To this end, it is apparently necessary to have a
representative of the Adams operations in algebraic $K$-theory, in terms of a
chain morphism, which commutes, at least up to a given homotopy, with the
representative of the Beilinson regulator ``$\ch$''. In \cite{FeliuAdams}, the author constructed a chain morphism representing the
Adams operations in algebraic $K$-theory tensored by $\Q$. In this paper, we
show that a slight modification of the construction in \cite{FeliuAdams} commutes strictly with ``$\ch$'', and
we deduce a pre-$\lambda$-ring structure for both $\widehat{K}_n(X)\otimes \Q$ and $\widehat{K}_n^T(X)\otimes \Q$.
The modification needs to be introduced in order to deal with the fact that the Koszul complex, when induced with its natural hermitian metrics,
does not have zero Bott-Chern form. A discussion on the Bott-Chern form of the Koszul complex is found in section $\S$\ref{bottchernkoszul}.

In order to work with $\widehat{K}_n^T(X)\otimes \Q$, we introduce the
\emph{modified homology groups}, which are the homological analogue of the
modified homotopy groups given by Takeda and the dual notion of the
\emph{truncated relative cohomology groups} defined by Burgos in
\cite{Burgos2}. We show that the homology groups modified by ``$\ch$'' give a
homological description of $\widehat{K}_n^T(X)\otimes \Q$.

\vspace{0.3cm} The paper is organized as follows.
In the first section we review the construction of the Beilinson regulator ``$\ch$'' given by Burgos and Wang in \cite{Burgos1}.
Next, we recall the definition of the arithmetic $K$-group of an arithmetic variety, $\widehat{K}_0(X)$, and proceed to the description of the higher arithmetic $K$-groups, both the Deligne-Soul\'e version and the Takeda one.
In the third section we introduce the modified homology groups and show that the Takeda higher arithmetic $K$-groups
admit a homological description after being tensored by $\Q$. Finally, the last section is devoted to the construction of Adams operations in higher arithmetic $K$-theory.

\vspace{0.5cm} \textbf{Notation.} If $A$ is an abelian group, we
denote $A_{\Q} := A\otimes {\Q}. $

We follow the conventions and definitions on (co)chain complexes and iterated (co)chain complexes
as given in \cite[$\S$2]{BurgosKuhnKramer}. All (co)chain complexes are of abelian groups.

If $(A_*,d_A)$ is a chain complex, we denote by $ZA_n$ the group of \emph{cycles} of degree $n$, that is, it consists of the elements $z\in A_n$ such that $d_A(z)=0$.
If $f:A_*\rightarrow B_*$ is a chain morphism, we denote by $s(f)_*$ the \emph{simple complex} associated to $f$. This is the same as the cone twisted by $-1$.

Given a chain complex $B_*$, let $\sigma_{> n}B_*$ be its \emph{b{\^e}te truncation} that is, the complex with
$$\sigma_{> n}B_r= \left\{ \begin{array}{ll} B_r & r > n, \\ 0 & r\leq n,
\end{array}\right.  $$
and differential induced by the differential of $B_*$.

\section{Higher Bott-Chern forms}
Burgos-Wang construction of the Beilinson regulator, given in \cite{Burgos1}, plays a key role in the definition of the higher arithmetic $K$-groups.
Using the chain complex of cubes, the transgression of vector bundles, and the Chern
character form of a vector bundle, they obtained a chain morphism
whose induced morphism in homology is the Beilinson regulator.
The construction is based on the definition of higher Bott-Chern forms.  These forms are the extension
to hermitian $n$-cubes of the Chern character form of a hermitian
vector bundle.

In this section we review this construction. For further details see the given reference or alternatively see
\cite[$\S$3.2]{BurgosKuhnKramer2}.

We focus the discussion on the case of smooth proper complex
varieties, since this will be the case in our applications.
Nevertheless, most of the constructions can be adapted to the
non-proper case by using hermitian metrics smooth at infinity. See
the original reference for details.

\subsection{Higher algebraic $K$-theory}\label{higherK}
Let $\mmP$ be a small exact category and let $K_n(\mmP)$ denote the $n$-th algebraic $K$-group of $\mmP$ in the sense of Quillen in \cite{Quillen}. Let $S_{\cdot}(\mmP)$ be the Waldhausen simplicial set, defined in \cite{Waldhausen}, which computes the higher algebraic $K$-groups of $\mmP$, that is, we have
 $$K_n(\mmP) \cong \pi_{i+1}(|S_{\cdot}(\mmP)|,\{0\}). $$
   Denote by $\partial_i,s_i$ the face and degeneracy maps, respectively, of $S_{\cdot}(\mmP)$ and let $\Z S_*(\mmP)$ be the Moore complex associated to the simplicial set $S_{\cdot}(\mmP)$.

Let $\langle 0,1,2 \rangle$ be the category associated to the ordered set
$\{0,1,2\}$ and let $\langle 0,1,2 \rangle^n$ be the $n$-th cartesian power.
Given a functor
$$\langle 0,1,2 \rangle^n \xrightarrow{E} \mmP,$$
the image of an $n$-tuple $\bj=(j_1,\dots,j_n)$ is denoted by $E^{\bj}$. For such a functor
 one defines its faces by
$$(\partial_i^kE)^{\bj}=E^{j_1...j_{i-1},k,j_{i},\dots,j_{n-1}},$$
for all $i\in \{1,\dots,n\},\ k\in \{0,1,2\},\ \bj\in \{0,1,2\}^{n-1}.$

\begin{dfn}
An  \emph{$n$-cube} $E$ over $\mmP$ is a functor
$$\langle 0,1,2 \rangle^n \xrightarrow{E} \mmP$$
such that for all
$\bj\in \{0,1,2\}^{n-1}$ and $i=1,\dots,n$ the sequence
\begin{equation}\label{shortexact}
(\partial_i^0E)^{\bj} \rightarrow
(\partial_i^1E)^{\bj} \rightarrow (\partial_i^2E)^{\bj}\end{equation} is a short exact
sequence of $\mmP$.
\end{dfn}
A functor $E$ as above is usually called a \emph{cube} and the sequences \eqref{shortexact} being exact make the cube be called an \emph{exact cube}. Since we will only consider exact cubes, the word exact is dropped from the terminology.

For every $n\geq 0$, let $C_n(\mmP)$ denote the set of $n$-cubes over $\mmP$.
We have defined face maps
\begin{eqnarray*}
\partial_i^k: C_n(\mmP) & \rightarrow & C_{n-1}(\mmP), \qquad i=1,\dots,n,\ k=0,1,2.
\end{eqnarray*}
There are as well degeneracy maps defined (see for instance \cite[$\S$3]{Burgos1})
\begin{eqnarray*}
s_i^k: C_{n-1}(\mmP) & \rightarrow & C_{n}(\mmP), \qquad i=1,\dots,n,\ k=0,1.
\end{eqnarray*}

If we write $\Z C_n(\mmP)$ to be the free abelian group on the $n$-cubes, the alternate sum of the face maps, $d=\sum_{i=^1}^n(-1)^i(\partial_i^0-\partial_i^1+\partial_i^2)$,
endows $\Z C_*(\mmP)$ with a chain complex structure with differential $d$.

Let
$$\Z D_n(\mmP)=\sum_{i=1}^{n} s_i^0(\Z C_{n-1}(\mmP))+ s_i^1(\Z C_{n-1}(\mmP))\subset \Z C_n(\mmP),$$
be the subgroup generated by the degenerate cubes (i.e., those that lie in the image of some degeneracy map).
The differential of $\Z C_*(\mmP)$ induces a differential on $\Z D_*(\mmP)$ making the inclusion arrow
$\Z D_*(\mmP)  \hookrightarrow \Z C_*(\mmP)$ a chain morphism. The quotient complex
$$\wZ C_*(\mmP)=\Z C_*(\mmP)/\Z D_*(\mmP)$$ is called the \emph{chain complex of cubes} in $\mmP$.

\medskip
\textbf{The $\Cub$ morphism. }
As shown in Wang's thesis \cite{Wang} and in \cite{McCarthy}, to every element $E\in S_{n}(\mmP)$ one can associate an $(n-1)$-cube $\Cub(E)$ satisfying the following property.
For $i=1,\dots,n-1$, we have
\begin{eqnarray}
\partial_i^{0}\Cub E &=& s_{n-2}^0\cdots s_{i}^0 \Cub \partial_{i+1}\cdots
\partial_n E, \nonumber \\
\partial_i^{1}\Cub E &=& \Cub \partial_{i} E,  \label{facescubes}  \\
\partial_i^{2}\Cub E &=& s_{i-1}^1\cdots s_{1}^1 \Cub \partial_{0}\cdots
\partial_{i-1} E. \nonumber
\end{eqnarray}

It follows from these equalities that $\Cub$ gives a chain morphism
$$ \Z S_{*}(\mmP)[-1]\xrightarrow{\Cub} \wZ C_*(\mmP).$$
The composition of the Hurewicz morphism with the morphism
induced by $\Cub$ in homology gives a morphism
$$\Cub : K_{n}(\mmP)=\pi_{n+1}(S_{\cdot}(\mmP)) \xrightarrow{\mathrm{Hurewicz}}
H_n(\Z S_{*}(\mmP)[-1]) \xrightarrow{\Cub} H_n(\wZ C_*(\mmP)).$$
Moreover, McCarthy  showed in \cite{McCarthy} that this morphism is an isomorphism over the field of rational numbers, that is, for all $n\geq
0$, the morphism
\begin{equation}\label{mccarthyiso}
K_n(\mmP)_{\Q} \xrightarrow{\Cub} H_n(\wZ C_*(\mmP),\Q)
\end{equation}
is an isomorphism.

\medskip
\textbf{Normalized complexes. }
Consider now the \emph{normalized chain complex} $N C_*(\mmP)$ introduced in \cite[$\S$1.5]{FeliuAdams}, whose $n$-th graded piece is given by
$$N_n C(\mmP) := \bigcap_{i=1}^{n} \ker\partial_i^0\cap \bigcap_{i=1}^{n}
\ker \partial_i^1\subset \Z C_n(\mmP),$$ and its differential is the one induced
by the differential of $\Z C_*(\mmP)$.  In loc. cit. it is shown that the composition
$$N_*C(\mmP) \hookrightarrow \Z C_*(\mmP) \twoheadrightarrow \wZ C_*(\mmP)  $$
is an isomorphism of chain complexes.

Let $NS_*(\mmP)$ be the normalized complex associated to the simplicial abelian group $\Z
S_{\cdot}(\mmP)$ given by
$$ N S_{n}(\mmP)=\bigcap_{i=1}^{n} \ker \partial_i,\quad n\geq 0,$$
and whose differential is $\partial_0$.
It follows from the relations in \eqref{facescubes} that the morphism $\Cub$ induces a chain morphism
$$
N S_{*}(\mmP)[-1] \xrightarrow{\Cub}  N C_*(\mmP).$$

\subsection{Chern character form}
In this subsection, all schemes are over $\C$.
As defined in \cite{Burgos2}, for every $p\geq 0$, let $\mmD^*(X,p)$ denote the Deligne complex of differential forms on $X$ computing Deligne-Beilinson cohomology groups with real coefficients twisted by $p$, $H_{\mmD}^{*}(X,\R(p))$.
We will write $\mmD^{2p-*}(X,p)$ for the chain complex associated to the cochain complex $\mmD^*(X,p)[2p]$.

Let $X$ be a smooth proper complex variety.
A \emph{hermitian vector bundle}
$\overline{E}=(E,h)$ is an algebraic vector bundle $E$ over $X$
together with a smooth hermitian metric on $E$. The reader is
referred to \cite{Wells} for details.

For every hermitian vector bundle $\overline{E}$, by the Chern-Weil formulae one defines a
closed differential form
$$\ch(\overline{E})\in \bigoplus_{p\geq 0}\mmD^{2p}(X,p),$$
representing the Chern character class
$\ch(E)=[\ch(\overline{E})]\in H_{dR}^{*}(X)$.
Although the class of $\ch(\overline{E})$ is independent of the
hermitian metric, the form depends on the particular choice of hermitian
metric.

The following properties are satisfied:
\begin{enumerate*}[$\rhd$] \item If $\overline{E}\cong \overline{F}$ is an
isometry of hermitian vector bundles, then
$ \ch(\overline{E}) = \ch(\overline{F}).$
\item Let $\overline{E}_1$ and $\overline{E}_2$ be two hermitian
vector bundles. If $\overline{E}_1\oplus \overline{E}_2$ and
$\overline{E}_1\otimes \overline{E}_2$ have the hermitian metrics
induced by those on $\overline{E}_1$ and $\overline{E}_2$, then
$$\ch(\overline{E}_1\oplus \overline{E}_2)= \ch(\overline{E}_1)+ \ch(\overline{E}_2), \quad \textrm{and}\quad \ch(\overline{E}_1\otimes \overline{E}_2) = \ch(\overline{E}_1)\wedge\ch(\overline{E}_2).$$
\end{enumerate*}

\subsection{Hermitian cubes}
Let $X$ be a smooth proper complex variety. Let $\mmP(X)$ be the
category of vector bundles over $X$. Let $\widehat{\mmP}(X)$ be
the category whose objects are the hermitian vector bundles over
$X$, and whose morphisms are given by
$$\Hom\nolimits_{\widehat{\mmP}(X)}((E,h),(E',h'))=\Hom\nolimits_{\mmP(X)}(E,E').$$
The category $\widehat{\mmP}(X)$ inherits an exact category
structure from that of $\mmP(X)$.

We fix a universe $\mmU$ so that $\widehat{\mmP}(X)$ is
$\mathcal{U}$-small for every smooth proper complex variety $X$.
Every vector bundle admits a smooth hermitian metric. It follows
that the forgetful functor
$$ \widehat{\mmP}(X) \rightarrow \mmP(X)$$
is an equivalence of categories with its quasi-inverse constructed
by choosing a hermitian metric for each vector bundle. Therefore,
the algebraic $K$-groups of $X$ can be computed in terms of the
category $\widehat{\mmP}(X)$.

Denote by $\widehat{S}_{\cdot}(X)$ the Waldhausen simplicial set
corresponding to the exact category $\widehat{\mmP}(X)$ and let
$\Z \widehat{C}_*(X)=\Z C_*(\widehat{\mmP}(X))$, $\wZ \widehat{C}_*(X)=\wZ C_*(\widehat{\mmP}(X))$ and $N\widehat{C}_*(X)=N C_*(\widehat{\mmP}(X))$. The cubes in the
category $\widehat{\mmP}(X)$ are called \emph{hermitian cubes}.

\medskip
\textbf{Hermitian cubes with canonical kernels.}
Let $\overline{E}$ be a hermitian vector bundle and let $F\subset
\overline{E}$ be an inclusion of vector bundles. Then $F$ inherits
a hermitian metric from the hermitian metric of $\overline{E}$. It
follows that there is an induced hermitian metric on the kernel of
a morphism of hermitian vector bundles. This allows to extend the definition of cubes with canonical kernels given in \cite{FeliuAdams}
to hermitian cubes in the following sense.

\begin{dfn}
Let $\overline{E}$ be a hermitian $n$-cube and let $g_i^0:\partial_i^0\overline{E}\rightarrow
\partial_i^1\overline{E}$ denote the morphism in the cube. We say that
$\overline{E}$ has \emph{canonical kernels} if for every
$i=1,\dots,n$ and $\bj\in \{0,1,2\}^{n-1}$, there is an inclusion
$(\partial_i^0\overline{E})^{\bj}\subset
(\partial_i^1\overline{E})^{\bj}$ of sets, the morphism
$$g_i^0:\partial_i^0\overline{E}\rightarrow
\partial_i^1\overline{E}$$ is the canonical inclusion of cubes and the metric on $\partial_i^0\overline{E}$
is induced by the metric of $\partial_i^1\overline{E}$ by means of
$g_i^0$.
\end{dfn}
The differential of a hermitian cube with canonical kernels is again a hermitian cube with canonical kernels.
Let  $\Z K\widehat{C}_*(X)$ denote the complex of hermitian cubes
with canonical kernels. As usual, the quotient of the complex of cubes with
canonical kernels by the degenerate cubes with canonical kernels
is denoted by $\wZ K\widehat{C}_*(X)$.

\begin{rem} Burgos and Wang \cite[Definition 3.5]{Burgos1} introduced the notion of \emph{emi-cubes}, in order
to define the morphism ``$\ch$''. With the notation of last
definition, the emi-cubes are those for which the metric on
$\partial_i^0\overline{E}$ is induced by the metric of
$\partial_i^1\overline{E}$, without the need of $g_i^0$ to be the
set inclusion. In loc. cit. the purpose was that the Chern form
of the transgression bundle associated to a cube defined a chain
morphism, and by the properties of ``$\ch$'' stated in $\S$1.2, this more relaxed condition was sufficient. Our more restrictive notion arises because we require the
transgression map given in \cite{FeliuAdams} to define a morphism, before being composed with the
Chern form (see below).
\end{rem}

\begin{lmm} There is a chain morphism
$$ \lambda:  \wZ \widehat{C}_*(X) \rightarrow \wZ K\widehat{C}_*(X).$$
\end{lmm}
\begin{proof} The morphism  $\lambda$ is defined in \cite{FeliuAdams} for cubes over the category of vector bundles (not necessarily hermitian).
The fact that the image by $\lambda$ of a hermitian cube is a hermitian cube with canonical kernels follows from \cite[Lemma 3.7]{Burgos1}.
\end{proof}

\subsection{The transgression bundle and the Chern character}
Let $\P^1$ be the complex projective line.
Let $x$ and $y$ be the global sections of the canonical bundle $\mathcal{O}(1)$ given
by the projective coordinates $(x:y)$ on $\P^1$. Let $X$ be a
complex variety and let  $p_0$ and $p_1$ be the projections from $X\times
\P^1$ to $X$ and $\P^1$ respectively. Then, for every vector bundle $E$ over $X$, we denote
$$E(k):=p_0^* E\otimes p_1^*\mathcal{O}(k),\quad \forall k.$$

The following definition is a variation of the original one from \cite[$\S$3]{Burgos1}.
\begin{dfn}\label{definiciotransgressio} Let
$$E: 0\rightarrow E^0 \xrightarrow{f^0} E^1 \xrightarrow{f^1} E^2 \rightarrow 0$$
be a short exact sequence. The \emph{first transgression bundle by
projective lines} of $E$, $\tr_1(E)$, is the kernel of the
morphism
\begin{eqnarray*}
E^1(1) \oplus E^2(1) &\rightarrow & E^2(2) \\ (a,b) &\mapsto &
f^1(a)\otimes x-b\otimes y.
\end{eqnarray*}
\end{dfn}

Let $E$ be an $n$-cube of vector bundles over $X$. We define  the \emph{first transgression}
of $E$ as the $(n-1)$-cube on $X\times (\P^1)^1$ given by
$$ \tr\nolimits_1(E)^{\bj}:=
\tr\nolimits_1(\partial_2^{j_2}\dots \partial_n^{j_n}E),\qquad \textrm{for all } \bj=(j_2,\dots,j_n)\in
 \{0,1,2\}^{n-1},$$
i.e. we take the transgression of the exact sequences in the first
direction. Since $\tr\nolimits_1$ is a functorial exact
construction, the \emph{$n$-th transgression bundle}
can be defined recursively as
$$\tr\nolimits_n(E)=\tr\nolimits_1\tr\nolimits_{n-1}(E)=
\tr\nolimits_1\stackrel{n}{\dots}\tr\nolimits_1(E).$$ It is a
vector bundle on $X\times (\P^1)^{n}$.

The Fubini-Study metric on $\P^1$ induces a metric on the line
bundle $\mathcal{O}(1)$. We denote by $\overline{\mathcal{O}(1)}$
the corresponding hermitian line bundle.

Then, given a hermitian $n$-cube $\overline{E}$, the
transgression bundle $\tr_n(\overline{E})$ has a hermitian metric
naturally induced by the metric of $\overline{E}$ and by the
metric of $\overline{\mathcal{O}(1)}$. If $\overline{E}$ is an $n$-cube with canonical kernels, then we have
\begin{eqnarray*}
 \tr\nolimits_n(\overline{E})|_{x_i=0} & = &  \tr\nolimits_{n-1}(\partial_i^1\overline{E}), \\
 \tr\nolimits_n(\overline{E})|_{y_i=0} & \cong &  \tr\nolimits_{n-1}(\partial_i^0\overline{E})
 \oplus^{\bot} \tr\nolimits_{n-1}(\partial_i^2\overline{E}),
\end{eqnarray*}
where $\cong$ is an isometry and $\oplus^{\bot}$ means the orthogonal direct sum.

Consider now  the differential form $W_n$ from \cite[$\S$6]{Burgos1}:
$$W_n=\frac{1}{2n!}\sum_{i=1}^n
(-1)^{i}S_n^i, $$ with
$$S_n^i= \sum_{\sigma\in \mathfrak{S}_n}(-1)^{\sigma}\log |z_{\sigma(1)}|^2
\frac{dz_{\sigma(2)}}{z_{\sigma(2)}}\wedge \cdots \wedge
\frac{dz_{\sigma(i)}}{z_{\sigma(i)}}\wedge \frac{d\bar
z_{\sigma(i+1)}}{\bar z_{\sigma(i+1)}}\wedge \cdots \wedge
\frac{d\bar z_{\sigma(n)}}{\bar z_{\sigma(n)}}.$$

\begin{thm}[(Burgos-Wang, \cite{Burgos1})] Let $X$ be a smooth proper complex variety.
\begin{enumerate}[(1)]
\item  The following map is a chain morphism
\begin{eqnarray}\label{ch3}
\wZ \widehat{C}_n(X)  &\xrightarrow{\ch}& \bigoplus_{p\geq 0}
\mmD^{2p-n}(X,p), \\
\overline{E} & \mapsto &  \ch\nolimits_n(\overline{E}):=
\frac{(-1)^n}{(2\pi i)^n} \int_{(\P^1)^n}
\ch(\tr\nolimits_n(\lambda(\overline{E})))\wedge W_n. \nonumber
\end{eqnarray}
\item The composition $$K_n(X)
\xrightarrow{\Cub} H_{n}(\wZ \widehat{C}_*(X)) \xrightarrow{\ch}
\bigoplus_{p\geq 0}H^{2p-n}_{\mmD}(X,\R(p)) $$ is the Beilinson
regulator.
\end{enumerate}
\end{thm}
\vist
The form $\ch\nolimits_n(\overline{E})$ is called the
\emph{Bott-Chern form} of the hermitian $n$-cube $\overline{E}$.

\begin{rem}
Observe that, by means of the isomorphism $N\widehat{C}_*(X)\cong
\wZ \widehat{C}_*(X)$, the Chern character is also
represented by the morphism
\begin{eqnarray*}
 N\widehat{C}_*(X) & \xrightarrow{\ch} &  \bigoplus_{p\geq 0}
\mmD^{2p-*}(X,p) \\ \overline{E} \in  N\widehat{C}_n(X) & \mapsto
&\ch\nolimits_n(\overline{E}).
\end{eqnarray*}
\end{rem}

\medskip
\textbf{Differential forms and projective lines.}
For some constructions in the sequel, it is convenient to factor the morphism ``$\ch$'' through a complex consisting of the Deligne complex
of differential forms on $X\times (\P^1)^n$. This construction was introduced in \cite{Burgos1}.

Over any base scheme, the cartesian
product of projective lines $(\P^1)^{\cdot}$ has a cocubical scheme structure.
Specifically, the coface and codegeneracy maps
\begin{eqnarray*}
\delta_j^i : (\P^1)^n &\rightarrow & (\P^1)^{n+1},\quad i=1,\dots,n,\ j=0,1, \\
\sigma^i : (\P^1)^n &\rightarrow & (\P^1)^{n-1} ,\quad
i=1,\dots,n,
\end{eqnarray*}
are defined as
\begin{eqnarray*}
  \delta_0^i(x_1,\dots,x_n) &=& (x_1,\dots,x_{i-1},(0:1),x_{i},\dots,x_n), \\
\delta_1^i(x_1,\dots,x_n) &=& (x_1,\dots,x_{i-1},(1:0),x_{i},\dots,x_n), \\
\sigma^i(x_1,\dots,x_n) &=& (x_1,\dots,x_{i-1},x_{i+1},\dots,x_n).
\end{eqnarray*}

Let $X$ be a smooth proper complex variety and fix $\P^1=\P_{\C}^1$.
The coface and codegeneracy maps induce,
for every $i=1,\dots,n$ and $l=0,1$, morphisms
\begin{eqnarray*}
\delta_i^l: \mathcal{D}^*(X\times (\P^1)^n,p) & \rightarrow &
\mathcal{D}^*(X\times (\P^1)^{n-1},p), \\
\sigma_i: \mathcal{D}^*(X\times (\P^1)^{n-1},p) & \rightarrow &
\mathcal{D}^*(X\times (\P^1)^n,p).
\end{eqnarray*}
Let $\mmD_{\P}^{*,*}(X,p)$  be the 2-iterated cochain complex
given by
$$\mmD_{\P}^{r,-n}(X,p) = \mmD^{r}(X\times(\P^1)^{n},p)$$ and differentials
$(d_{\mmD},\delta=\sum_{i=1}^n(-1)^i (\delta_i^0-\delta_i^1))$ and
denote by $\mmD_{\P}^{*}(X,p)$ the associated simple complex.

Let $(x:y)$ be homogeneous coordinates in $\P^1$ and consider
$$h =-\frac{1}{2} \log \frac{(x-y)\overline{(x-y)}}{x\bar{x}+y\bar{y}}.$$ It defines a
function on the open set $\P^1 \setminus \{1\}$,  with logarithmic
singularities along $1$. Consider the differential
$(1,1)$-form
$$\omega = d_{\mmD} h \in \mmD^2(\P^1,1).$$
This is a smooth form all over $\P^1$ representing the class of
the first Chern class of the canonical bundle of $\P^1$,
$c_1(\mathcal{O}_{\P^1}(1))$.

For every $n$, denote by $\pi:X\times (\P^1)^n\rightarrow X$ the projection onto $X$ and
let $p_i:X\times (\P^1)^n\rightarrow \P^1$ be the projection onto
the $i$-th projective line. Denote, for $i=1,\dots,n$,
\begin{eqnarray*}
\omega_i & = & p_i^*(\omega)\in \mmD^{2}(X\times (\P^1)^{n},1)
\end{eqnarray*}

Let
\begin{equation}
D_n^r = \sum_{i=1}^n \sigma_i(\mmD^{r}(X\times (\P^1)^{n-1},p)),
\end{equation}
be the complex of degenerate elements and let $\mmW_n^*$ be the
subcomplex of $\mmD^*(X\times (\P^1)^n,p)$ given by
\begin{equation}
\mathcal{W}_n^r = \sum_{i=1}^n \omega_i\wedge
\sigma_i(\mmD^{r-2}(X\times (\P^1)^{n-1},p-1)).
\end{equation}
This complex is meant to kill the cohomology classes coming from
the projective lines.
We define the $2$-iterated complex
$$\wmD_{\P}^{r,-n}(X,p):=\wmD^r(X\times (\P^1)^n,p):=\frac{\mmD^r(X\times (\P^1)^n,p)}{D_n^r + \mmW_n^r}$$
and denote by $\wmD_{\P}^{*}(X,p)$ the associated simple complex.

\begin{prp}\label{qisoproj} The natural map
$$\mmD^*(X,p)=\wmD_{\P}^{*,0}(X,p) \xrightarrow{i}  \wmD_{\P}^*(X,p) $$
is a quasi-isomorphism.
\end{prp}
\begin{proof} The proof is analogous to the proof of \cite[Lemma 1.3]{Burgos1}. It follows from an spectral sequence argument together with the fact that, by the Dold-Thom isomorphism in
Deligne-Beilinson cohomology,
\begin{equation}\label{spectral2}
H^r(\wmD^{*}(X\times (\P^1)^n,p))=0\quad  \forall n>0.
\end{equation}
\end{proof}

For the next proposition, $\bullet$ denotes the product in the Deligne complex as described in \cite[$\S$3]{Burgos2}.
\begin{prp}[(Burgos-Wang, \cite{Burgos1})]\label{qinverse}
There is a quasi-isomorphism of complexes
$$\wmD_{\P}^*(X,p) \xrightarrow{\varphi} \mmD^{*}(X,p),$$
given by
$$ \alpha \in \mmD^{r}(X\times (\P^1)^n,p) \mapsto
\pi_*(\alpha\bullet W_n)=\left\{\begin{array}{ll} \frac{1}{(2\pi
i)^{n}}\int_{(\P^1)^n} \alpha\bullet W_n & n>0, \\
\alpha & n=0. \end{array} \right.$$
This morphism is the quasi-inverse of the quasi-isomorphism $i$ of Proposition \ref{qisoproj}.
\end{prp}
\vist

\begin{prp}[(Burgos-Wang, \cite{Burgos1})]
The map
\begin{eqnarray*}
\wZ \widehat{C}_n(X)  &\xrightarrow{\ch}& \bigoplus_{p\geq 0}
\wmD^{2p-n}_{\P}(X,p), \\
\overline{E} & \mapsto &  \ch(\tr\nolimits_n(\lambda(\overline{E})))
\end{eqnarray*}
is a chain morphism. Therefore, the morphism $\ch$ of \eqref{ch3} factors through the complex
$\wmD^{2p-*}_{\P}(X,p)$ in the form
$$  \wZ \widehat{C}_n(X) \xrightarrow{\ch}  \bigoplus_{p\geq 0}
\wmD^{2p-n}_{\P}(X,p) \xrightarrow{\varphi} \bigoplus_{p\geq 0} \mmD^{2p-*}(X,p).$$
\end{prp}
\vist

\section{Higher arithmetic K-theory}
In this section we focus on the definition of the higher arithmetic $K$-groups of an arithmetic variety.
We start by discussing the extension of the Chern character on complex varieties to arithmetic varieties. Then, we recall the
definition of the arithmetic $K$-group given by Gillet and Soul{\'e} in \cite{GilletSouleClassesII}. Finally, the last two sections
are devoted to review the two definitions of  higher arithmetic $K$-theory.

Following \cite{GilletSouleIHES}, by an \emph{arithmetic variety} we mean a regular
quasi-projective scheme over an arithmetic ring. In this section we restrict ourselves to proper arithmetic
varieties over the arithmetic ring $\Z$. Note, however, that most
of the results are valid under the less restrictive hypothesis of
the variety being proper over $\C$. Moreover, one could extend the
definition of higher arithmetic $K$-groups, $\widehat{K}_n(X)$, to
quasi-projective varieties, by considering vector bundles with
hermitian metrics smooth at infinity and the complex of
differential forms $\wmD_{\P}^{*}(X,p)$ introduced in the previous section (and by considering differential forms with logarithmic
singularities).

\subsection{Chern character for arithmetic varieties}

If $X$ is an arithmetic variety over $\Z$, let $X(\C)$ denote the
associated complex variety, consisting of the $\C$-valued points
on $X$. Let $F_{\infty}$ denote the complex conjugation on $X(\C)$
and $X_{\R}=(X(\C),F_{\infty})$ the associated real variety.

The \emph{real Deligne-Beilinson cohomology of $X$} is
defined as the cohomology of $X_{\R}$, i.e.
$$H_{\mmD}^{n}(X,\R(p)) = H_{\mmD}^{n}(X_{\R},\R(p))=H_{\mmD}^n(X(\C),\R(p))^{\overline{F_{\infty}^*}=id}. $$
It is computed as the cohomology of the real Deligne complex:
$$\mmD^n(X,p)= \mmD^n(X_{\R},p)= \mmD^n(X(\C),p)^{\overline{F_{\infty}^*}=id},$$
that is, we have
$$H^n_{\mmD}(X,\R(p))\cong H^n(\mmD^n(X,p),d_{\mmD}).$$

\begin{dfn}
Let $X$ be a proper arithmetic variety over $\Z$. A
\emph{hermitian vector bundle} $\overline{E}$ over $X$ is a pair
$(E,h)$, where $E$ is a locally free sheaf on $X$ and where $h$ is
a $F_{\infty}^*$-invariant smooth hermitian metric on the associated
vector bundle $E(\C)$ over $X(\C)$.
\end{dfn}
Let
$\widehat{\mmP}(X)$ denote the category of hermitian vector
bundles over $X$. The simplicial set $\widehat{S}_{\cdot}(X)$ and
the chain complexes $\Z \widehat{C}_*(X)$, $\wZ \widehat{C}_*(X)$
and $N\widehat{C}_*(X)$ are defined accordingly.

If $\overline{E}$ is a hermitian vector bundle over $X$, the Chern
character form $\ch(\overline{E})$ is $\overline{F_{\infty}^*}$-invariant.
Therefore
$$\ch(\overline{E}) \in \bigoplus_{p\geq 0}\mmD^{2p}(X,p).$$
It follows that the chain morphism of \eqref{ch3} gives a chain morphism
$$\Z \widehat{S}_{*}(X)[-1] \xrightarrow{\Cub} \wZ \widehat{C}_{*}(X)
 \xrightarrow{\ch} \bigoplus_{p\geq 0} \mmD^{2p-*}(X,p). $$

\subsection{Arithmetic $K_0$-group}
In \cite[$\S$6]{GilletSouleClassesII}  Gillet and Soul{\'e} defined the
arithmetic $K_0$-group of an  arithmetic variety, denoted by
$\widehat{K}_0(X)$. We give here a slightly different presentation
using the Deligne complex of differential forms and the
differential operator $-2\partial\bar{\partial}$.

Let $X$ be an arithmetic variety and let $\widetilde{\mmD}^{
*}(X,p)=\mmD^{ *}(X,p)/\im d_{\mmD}$. Consider pairs
$(\overline{E},\alpha)$, where $\overline{E}$ is a hermitian
vector bundle over $X$ and where $\alpha \in \bigoplus_{p\geq
0}\wmD^{2p-1}(X,p)$ is a differential form. Then,
$\widehat{K}_0(X)$ is the quotient of the free abelian group
generated by these pairs by the subgroup generated by the sums
$$(\overline{E}^0,\alpha_0)+(\overline{E}^{2},\alpha_2) -(\overline{E}^1, \alpha_0+\alpha_2-\ch(\overline{E})),  $$
for every exact sequence of hermitian vector bundles over $X$,
$$\overline{E}: 0\rightarrow \overline{E}^0 \rightarrow \overline{E}^1 \rightarrow \overline{E}^2 \rightarrow 0, $$
and every $\alpha_0,\alpha_2 \in \bigoplus_{p\geq 0}\wmD^{2p-1}(X,p)$.

Among other properties, this group fits in an exact sequence
$$K_{1}(X) \xrightarrow{\ch}  \bigoplus_{p\geq 0} \widetilde{\mmD}^{2p-1}(X,p)\xrightarrow{a} \widehat{K}_0(X)
\xrightarrow{\zeta} K_0(X) \rightarrow 0$$ (see
\cite{GilletSouleClassesII} for details).

Gillet and Soul{\'e} in \cite{GilletSouleClassesII}, together with R{\"o}ssler in
\cite{Roessler1}, showed that there is a $\lambda$-ring structure
on $\widehat{K}_0(X)$.

\subsection{Deligne-Soul{\'e} higher arithmetic $K$-theory}
Although  there is no
reference in which the theory is developed, it has been suggested
by Deligne and Soul{\'e}  (see \cite[$\S$III.2.3.4]{Soule2}  and
\cite[Remark 5.4]{Delignedeterminant}) that the higher arithmetic
$K$-theory should be obtained as the homotopy groups of the
homotopy fiber of a representative of the Beilinson regulator. We
sketch here the construction, in order to show that Adams
operations can be defined.

Consider the b\^ete truncation at $n>0$ of the complex $\mmD^{2p-*}(X,p)$, denoted by $\sigma_{>0}\mmD^{2p-*}(X,p)$.
Let
$$ \widehat{\ch}: \wZ \widehat{C}_n(X) \rightarrow \bigoplus_{p \geq 0} \sigma_{>0}\mmD^{2p-n}(X,p), $$
be the composition of $\ch:\wZ \widehat{C}_n(X) \rightarrow \bigoplus_{p \geq 0} \mmD^{2p-n}(X,p)$ with the natural map
$$\bigoplus_{p \geq 0}  \mmD^{2p-n}(X,p)\rightarrow \bigoplus_{p \geq 0} \sigma_{>0}\mmD^{2p-n}(X,p).$$

Let $\mmK({\cdot})$ be the Dold-Puppe functor from the category of
chain complexes of abelian groups to the category of simplicial
abelian groups (see \cite{DoldPuppe}). Consider the morphism
$$\mmK(\widehat{\ch}): \widehat{S}_{\cdot}(X) \rightarrow \mmK_{\cdot}(\Z \widehat{S}_{*}(X))\xrightarrow{\Cub}
\mmK( \wZ \widehat{C}_*(X)) \xrightarrow{\widehat{\ch}}
\mmK\Big(\bigoplus_{p\geq 0}\sigma_{>0}\mmD^{2p-*}(X,p)\Big),$$ and denote by
$|\mmK(\widehat{\ch})|$  the morphism induced on the realization
of the simplicial sets.

\begin{dfn}\label{k0}
For every $n\geq 0$,  the \emph{(Deligne-Soul{\'e}) higher
arithmetic $K$-group} of $X$ is defined as
$$\widehat{K}_n(X) = \pi_{n+1}(\mathrm{Homotopy \ fiber\ of \  }|\mmK(\widehat{\ch})|). $$
\end{dfn}

\begin{prp}\label{Delignesouleprops} Let $X$ be a proper arithmetic variety.
Then,
\begin{enumerate}[(i)]
\item The group $\widehat{K}_0(X)$ as defined in Definition \ref{k0} agrees with the arithmetic $K$-group
defined by Gillet and Soul{\'e} in \cite{GilletSouleClassesII}.
\item Let $s(\widehat{\ch})$ denote the simple complex associated to the
chain morphism $\widehat{\ch}$.
If $n>0$, there is an isomorphism $\widehat{K}_n(X)_{\Q}
\cong H_n(s(\widehat{\ch}),\Q).$ \item There is a long exact
sequence
$$ \cdots \rightarrow K_{n+1}(X)  \xrightarrow{\ch} H_{\mmD}^{2p-n-1}(X,\R(p))
\xrightarrow{a} \widehat{K}_n(X)  \xrightarrow{\zeta} K_n(X)
 \rightarrow \cdots $$
 with end
 $$   \cdots \rightarrow K_{1}(X)  \xrightarrow{\ch} \wmD^{2p-1}(X,p)
\xrightarrow{a} \widehat{K}_0(X) \xrightarrow{\zeta} K_0(X)
 \rightarrow 0. $$
 \end{enumerate}
 \end{prp}
 \begin{proof}
The first and third statements follow by definition. The second
statement follows from the isomorphism of
\eqref{mccarthyiso} together with the following well-known fact (see \cite{FeliuThesis} for a proof):
\begin{lmm}
Let $(A_*,d_A)$, $(B_*,d_B)$ be two
chain complexes. Let $f:A_*\rightarrow B_*$ be a chain morphism and let
$\mmK(f):\mmK_{\cdot}(A) \rightarrow \mmK_{\cdot}(B)$ be the induced
morphism. Let $\HoFib(f)$ denote the homotopy fiber of the topological realization of $\mmK(f)$.
Then, for every $n\geq 1$, there is an isomorphism
$$ \pi_n(\HoFib(f)) \rightarrow H_n(s_*(f))$$
such that the following diagram is commutative:
$$\xymatrix@C=20pt{
 \pi_{n+1}(\mmK_{\cdot}(B))\ar[d]_{\cong} \ar[r]
&\pi_{n}(\HoFib(f))\ar[d] \ar[r] &\pi_{n}(\mmK_{\cdot}(A))
\ar[d]_{\cong}
\\  H_{n+1}(B_*) \ar[r] & H_n(s(f)) \ar[r] & H_n(A_*).
}$$
\end{lmm}

\end{proof}

In section \ref{rationaltakedakgroups}, we will endow
$\bigoplus_{n\geq 0}\widehat{K}_n(X)$ with a product structure,
induced by the product structure defined by Takeda on his higher arithmetic $K$-groups.

\subsection{Takeda higher arithmetic $K$-theory}
In this section we recall the definition of higher arithmetic
$K$-groups given by Takeda in \cite{Takeda}. He first develops a
theory of homotopy groups modified by a suitable chain morphism
$\rho$. As a particular case, the higher arithmetic $K$-groups are
given by the homotopy groups of $\widehat{S}_{\cdot}(X)$ modified
by the Chern character morphism ``$\ch$''.

Let T be a pointed CW-complex and let $C_*(T)$ be its cellular complex (see, for instance,
\cite{Hatcher}).
Let $(W_*,d)$ be a chain complex and denote
$\widetilde{W}_*=W_*/\im d$. Suppose that we are given a chain morphism
$\rho:C_*(T)\rightarrow W_*$. Consider pairs $(f,\omega)$ where
\begin{enumerate*}[$\rhd$]
\item $f:S^{n}\rightarrow T$ is a pointed cellular map, and, \item
$\omega\in \widetilde{W}_{n+1}$.
\end{enumerate*}

Let $I$ be the closed unit interval $[0,1]$ with the usual
CW-complex structure. Two pairs $(f,\omega)$ and $(f',\omega')$
are \emph{homotopy equivalent} if there exists a pointed cellular
map
$$h: S^n\times I/\{*\}\times I \rightarrow T $$
such that the following conditions hold:
\begin{enumerate}[(1)]
\item $h$ is a topological homotopy between $f$
 and $f'$, i.e.
 $$h(x,0)=f(x),\quad\textrm{and}\quad h(x,1)=f'(x). $$
 \item Let $[S^n\times I]\in C_{n+1}(S^n\times I)$ denote the fundamental chain of $S^n\times
 I$. Then,
 $$\omega'-\omega= (-1)^{n+1} \rho \big(h_*([S^n\times I])\big). $$
\end{enumerate}
Being homotopy equivalent is an equivalence relation, which we
denote by $\sim$. Then, for every $n$, the \emph{modified homotopy
group} $\widehat{\pi}_n(T,\rho)$ is defined to be the  set of all
homotopy classes of pairs as above. Takeda proves that these are in fact
abelian groups.

The higher arithmetic $K$-groups of a proper arithmetic variety $X$, as defined by Takeda, are given as the modified homotopy groups of the Waldhausen simplicial set $\whP(X)$ modified by
the representative of the Beilinson regulator ``$\ch$'' given in the previous section.

Let $X$ be a proper arithmetic variety over $\Z$. Let
$|\widehat{S}_{\cdot}(X)|$ denote the geometric realization of the
simplicial set $\widehat{S}_{\cdot}(X)$. It follows that
$|\widehat{S}_{\cdot}(X)|$ is a CW-complex.

Let $\widehat{D}^s_*(X)\subset \Z \widehat{S}_{*}(X)$ be the
complex generated by the degenerate simplices of
$\widehat{S}_{{\cdot}}(X)$. Since the cellular complex
$C_{*}(|\widehat{S}_{\cdot}(X)|)$ is naturally isomorphic to the
complex $\Z \widehat{S}_{*}(X)/\widehat{D}^s_*(X)$, we will in the sequel identify these two complexes by this isomorphism.

As shown in \cite[Theorem 4.4]{Takeda}, the map $\ch \circ \Cub$
maps the degenerate simplices of $\widehat{S}_{\cdot}(X)$ to zero.
It follows that there is a well-defined chain morphism
$$\ch:C_*(|\widehat{S}_{\cdot}(X)|)[-1] \rightarrow  \bigoplus_{p\geq 0} \mmD^{2p-*}(X,p).$$

\begin{dfn}[(Takeda)]
Let $X$ be a proper arithmetic variety over $\Z$. For every $n\geq
0$, the \emph{higher arithmetic $K$-group} of $X$,
$\widehat{K}_{n}^{T}(X)$,  is defined by
\begin{eqnarray*}
\widehat{K}_{n}^{T}(X) & = & \widehat{\pi}_{n+1}(
|\widehat{S}_{\cdot}(X)|,\ch) \\ & = & \Big\{(f:S^{n+1}\rightarrow
|\widehat{S}_{\cdot}(X)|,\ \omega)\ \big| \ \omega \in
\bigoplus_{p\geq 0} \widetilde{\mmD}^{2p-n-1}(X,p)\Big\}
\Big/\sim.
\end{eqnarray*}
\end{dfn}

Takeda proves the following results:
\begin{enumerate}[(i)]
\item For every $n\geq 0$, $\widehat{K}_{n}^{T}(X)$ is a group.
\item For every $n\geq 0$, there is a short exact sequence
$$K_{n+1}(X) \xrightarrow{\ch}  \bigoplus_{p\geq 0} \widetilde{\mmD}^{2p-n-1}(X,p)\xrightarrow{a} \widehat{K}^{T}_n(X)
\xrightarrow{\zeta} K_n(X) \rightarrow 0.$$ The morphisms
$a,\zeta$ are defined by $ a(\alpha)=[(0,\alpha)]$ and $\zeta([(f,\alpha)])=[f].$
\item There is a characteristic class
$$\widehat{K}^{T}_n(X)\xrightarrow{\ch} \bigoplus_{p\geq 0}{\mmD}^{2p-n}(X,p), $$
given by
$$  \ch([(f, \alpha)])=\ch(f_*(S^n))+d_{\mmD} \alpha.$$
 \item
$\widehat{K}^{T}_0(X)$ is isomorphic to the arithmetic $K$-group
defined by Gillet and Soul{\'e} in \cite{GilletSouleClassesII}.
\item There is a graded product on $\widehat{K}^{T}_*(X)$,
commutative up to $2$-torsion. Therefore,
$\widehat{K}^{T}_*(X)_{\Q}$ is endowed with a graded commutative
product. \item There e\-xist pull-back for ar\-bi\-trary morphisms and
push-forward for smooth and projective morphisms. A projection
formula is also proved.
\end{enumerate}

\begin{lmm}[(\cite{Takeda},Cor. 4.9)] Let $X$ be a proper arithmetic variety over $\Z$.
Then, for every $n\geq 0$, there is a canonical isomorphism
$$\widehat{K}_n(X)\cong
\ker \big(\ch: \widehat{K}_n^{T}(X)\rightarrow
\bigoplus_{p\geq 0} \mmD^{2p-n}(X,p) \big). $$
\end{lmm}
\vist

\section{Rational higher arithmetic $K$-groups}
By parallelism with the algebraic situation, it is natural to
expect that the higher arithmetic $K$-groups tensored by $\Q$ can
be described in homological terms. In Proposition
\ref{Delignesouleprops}, we saw that the Deligne-Soul{\'e} higher
arithmetic $K$-groups are isomorphic to the homology groups of the
simple complex associated to the Beilinson regulator ``$\ch$'',
after tensoring by $\Q$. In this section we show that the higher arithmetic $K$-groups given by Takeda admit also, after tensoring by $\Q$, a homological
description. We prove that $\widehat{K}^{T}_n(X)_{\Q}$ can be
obtained considering a variant of the complex of cubes, together
with what we call modified homology groups.

\subsection{Modified homology groups} We briefly describe
here the analogue, in a homological context, of the modified
homotopy groups given by Takeda in \cite{Takeda}.
The modified homology groups are the dual notion of the
\emph{truncated relative cohomology groups} defined by Burgos in
\cite{Burgos2}, as one can observe comparing both definitions and
the satisfied properties. These groups appear naturally in other
contexts. For instance, one can express the description of
hermitian-holomorphic Deligne cohomology given by Aldrovandi in
\cite[$\S$2.2]{Aldrovandi}, in terms of modified homology groups.

Let $(A_*,d_A)$ and $(B_*,d_B)$ be two chain complexes and let
$A_*\xrightarrow{\rho} B_*$ be a chain morphism. If
$\widetilde{B}_*=B_*/\im d_B$,  consider pairs
$$(a,b)\in A_n\oplus \widetilde{B}_{n+1},\quad \textrm{with } d_Aa=0.$$

We define an equivalence relation as follows. We say that
$(a,b)\sim (a',b') $ if, and only if, there exists $h\in A_{n+1}$,
such that
$$d_Ah=a-a',\quad \textrm{and }\quad \rho (h) =b -b'.$$

\begin{dfn} Let $(A_*,d_A),(B_*,d_B)$ be two chain complexes and let $\rho:A_*\rightarrow B_*$ be a chain morphism.
 For every $n$, the \emph{$n$-th modified homology group of $A_*$ with respect to
 $\rho$} is defined as
$$ \widehat{H}_{n}(A_*,\rho) :=\{ (a,b)\in ZA_n\oplus \widetilde{B}_{n+1}\}/ \sim. $$
\end{dfn}
Observe that the group $\widehat{H}_{n}(A_*,\rho)$
can be rewritten as
$$  \widehat{H}_{n}(A_*,\rho) =\frac{\big\{ (a,b)\in ZA_n\oplus B_{n+1}\big\}}{ \big\{(0,d_Bb),(d_Aa,\rho(a)),
a\in A_{n+1},b\in B_{n+2} \big\}}. $$
The class of a pair $(a,b)$ in $\widehat{H}_{n}(A_*,\rho)$ is
denoted by $[(a,b)]$.

These modified homology groups can be seen as the homology groups
of the simple of $\rho$ truncated appropriately.
Let $\rho_{> n}$ be the composition of $\rho:A_*\rightarrow B_*$ with the
canonical morphism $B_*\rightarrow \sigma_{>n}B_*.$
Then, it follows from the definition that
$$H_r(s(\rho_{> n}))= \left\{ \begin{array}{ll} H_r(s(\rho)) & r>n, \\ \widehat{H}_{n}(A_*,\rho) & r=n, \\
H_r(A_*) & r<n.
\end{array}\right. $$
Observe that, for every $n$, there are well-defined
morphisms
$$
\begin{array}{rclcrcl}
\widetilde{B}_{n+1} & \xrightarrow{a} & \widehat{H}_{n}(A_*,\rho), & \quad &  b & \mapsto & [(0,-b)],\\
 \widehat{H}_{n}(A_*,\rho) & \xrightarrow{\zeta} & H_n(A_*), &\quad &
 [(a,b)]&
 \mapsto & [a], \\
 \widehat{H}_{n}(A_*,\rho) & \xrightarrow{\rho} & Z B_n & \quad &  [(a,b)] & \mapsto &
\rho(a)-d_B(b).
\end{array}
$$

The following proposition  is the homological analogue of Theorem
3.3 together with Proposition 3.9 of \cite{Takeda} and the dual of
Propositions 4.3 and 4.4 of \cite{Burgos2}.

\begin{prp}\label{arithlong}
\begin{enumerate}[(i)] \item Let $\rho:A_*\rightarrow B_*$ be a chain
morphism. Then, for every $n$, there are exact sequences
\begin{enumerate}[(a)]
\item $ 0\rightarrow H_n(s_*(\rho)) \rightarrow \widehat{H}_{n}(A_*,\rho) \xrightarrow{\rho}  Z B_n \rightarrow H_{n-1}(s_*(\rho)). $
\item $H_{n+1}(A_*) \xrightarrow{\rho} \widetilde{B}_{n+1}  \xrightarrow{a}
 \widehat{H}_{n}(A_*,\rho) \xrightarrow{\zeta}  H_n(A_*)
\rightarrow 0.$
\end{enumerate}
\item Assume that there is a commutative square of  chain
complexes
$$\xymatrix{
A_* \ar[r]^{\rho}\ar[d]_{f_1} & B_* \ar[d]^{f_2} \\ C_* \ar[r]_{\rho'} & D_*. }
$$
Then, for every $n$, there is an induced morphism
$$ \widehat{H}_{n}(A_*,\rho) \xrightarrow{f}  \widehat{H}_n(C_*,\rho') \qquad \ [(a,b)]  \mapsto
[(f_1(a),f_2(b))]. $$
\item If $f_1$ is a quasi-isomorphism and $f_2$ is an isomorphism, then $f$ is
an isomorphism.
\end{enumerate}
\end{prp}
\begin{proof}
The exact sequences follow from the long exact sequences associated to the
following short exact sequences:
\begin{align*}
0\rightarrow B_*/\sigma_{>n}B_*[-1]  \rightarrow s_*(\rho)
\rightarrow s_*(\rho_{>n}) \rightarrow 0, &
\\
0\rightarrow \sigma_{>n}B_*[-1] \rightarrow s_*(\rho_{>n})
\rightarrow A_* \rightarrow 0. &
\end{align*}
The second and the third statements are left to the reader.
\end{proof}

\begin{crl} For every $n$, there is a canonical isomorphism
$$H_n(s_*(\rho)) \cong_{can} \ker(\widehat{H}_{n}(A_*,\rho) \xrightarrow{\rho} B_n). $$
\end{crl}
\vist

\subsection{Takeda  arithmetic $K$-theory with rational coefficients}\label{rationaltakedakgroups}

We want to give a homological description of the rational Takeda
arithmetic $K$-groups. Since these groups fit in the exact
sequences
$$K_{n+1}(X)_{\Q} \xrightarrow{\ch}  \bigoplus_{p\geq 0} \widetilde{\mmD}^{2p-n-1}(X,p)
\xrightarrow{a} \widehat{K}^{T}_n(X)_{\Q} \xrightarrow{\zeta}
K_n(X)_{\Q} \rightarrow 0,$$ comparing it to \ref{arithlong}(i)(b), it is natural to expect that the
modified homology groups associated to the Beilinson regulator
``$\ch$'' give the desired description.

Therefore, consider the modified homology groups
$\widehat{H}_{n}(\wZ \widehat{C}_{*}(X),\hch)$ associated to the
chain map
$$\wZ \widehat{C}_{*}(X)  \xrightarrow{\ch}
 \bigoplus_{p\geq 0} \mmD^{2p-*}(X,p)$$
given in \eqref{ch3}. We want to see that there is an isomorphism
\begin{equation}
\widehat{K}_n^T(X)_{\Q} \cong \widehat{H}_{n}(\wZ
\widehat{C}_{*}(X),\ch)_{\Q}.\label{isocong}
\end{equation}

In order to prove this fact, considering the long exact sequences
associated to $\widehat{K}_n^T(X)_{\Q}$, to $\widehat{H}_{n}(\wZ
\widehat{C}_{*}(X),\ch)_{\Q}$ and the five lemma, it would be
desirable to have a factorization of the morphism ``$\ch$'' by
$\Cub$ in the form
{\small $$ \xymatrix{\Z \widehat{S}_*(X) \ar[dr] \ar[r]^{\Cub} & \wZ \widehat{C}_*(X) \ar[r]^(0.35){\ch} & \bigoplus_{p\geq 0}\mmD^{2p-*}(X,p)
\\  & C_{*}(|\widehat{S}_{\cdot}(X)|)[-1] \cong \widehat{S}_{*}(X)/\widehat{D}^s_*(X)[-1] \ar@{.>}[u]^{\Cub} \ar[ur]_{\ch}}$$}

Let $\mmP$ be a small exact category. If $\tau_i\in
\mathfrak{S}_n$ is the permutation that interchanges $i$ with $i+1$, then for
every $E\in S_{n}(\mmP)$ one has
\begin{eqnarray}
\Cub(s_0E)&=& s_1^1 \Cub(E), \nonumber \\ \Cub(s_nE) &=& s_n^0\Cub(E), \label{sifaces} \\
\Cub(s_iE)&=& \tau_i \Cub(s_iE), \qquad i=1,\dots,n-1. \nonumber
\end{eqnarray}
(See \cite[Lemma 4.1]{Takeda}).
It follows that the dotted arrow $\Cub$ of last diagram
$$C_{*}(|\widehat{S}_{\cdot}(X)|)[-1]\cong \Z \widehat{S}_{*}(X)/\widehat{D}^s_*(X)[-1] \xrightarrow{\Cub} \wZ \widehat{C}_*(X)$$
does not exist, since the image by $\Cub$ of a degenerate simplex
in $S_{n}(\mmP)$ is not necessarily a degenerate cube.

Therefore, in order to prove  \eqref{isocong}, we should find a
new complex, $\wZ \widehat{C}_{*}^s(X)$, quasi-isomorphic to the
complex of hermitian cubes, admitting a factorization of ``$\ch$''
of the form:
$$ \xymatrix{\Z \widehat{S}_*(X) \ar[d] \ar[r]^{\Cub} & \wZ \widehat{C}_*(X)\ar[d]_{\sim} \ar[r]^(0.4){\ch} & \bigoplus_{p\geq 0}\mmD^{2p-*}(X,p)
\\  C_{*}(|\widehat{S}_{\cdot}(X)|)[-1] \ar@{.>}[r]^(0.57){\Cub}  &  \wZ \widehat{C}_{*}^s(X) \ar[ur]_{\ch}}$$
In this way, we divide the proof in two steps: to prove that there
is an isomorphism $\widehat{H}_n( \wZ \widehat{C}_*(X),\ch) \cong
\widehat{H}_n( \wZ \widehat{C}_*^s(X),\ch)$, and then that $
\widehat{H}_n( \wZ \widehat{C}_*^s(X),\ch)_{\Q}\cong
\widehat{K}_n^T(X)_{\Q}$. This will be shown in Theorem
\ref{arithrational}, once this factorization of $\Cub$ is
obtained.

\medskip
\textbf{Factorization of $\Cub$. }
Takeda factors the morphism ``$\ch$'' through a quotient of the complex of cubes as follows.
Consider the complex of cubes $\widehat{T}_n(X)\subseteq \Z
\widehat{C}_{n}(X)$, generated by the  $n$-cubes $\overline{E}$
such that $\tau_i\overline{E}=\overline{E}$ for some index $i$. In
the proof of Theorem 4.4 in \cite{Takeda}, Takeda shows that if
$\overline{E}\in \widehat{T}_n(X)$, then $\ch(\overline{E})=0$.
Hence the morphism ``$\ch$'' is zero on the degenerate simplices
in $\Z \widehat{S}_*(X)$. It follows that $\ch$ factorizes as
{\small
$$C_{*}(|\widehat{S}_{\cdot}(X)|)[-1] \xrightarrow{\cong} \Z
\widehat{S}_{*}(X)/\widehat{D}^s_*(X)[-1] \xrightarrow{\Cub} \wZ
\widehat{C}_{*}(X) / \widehat{T}_*(X) \xrightarrow{\ch}
\bigoplus_{p\geq 0} \mmD^{2p-*}(X,p).
$$}

However, the complex $\wZ \widehat{C}_{*}(X) / \widehat{T}_*(X)$ is not
quasi-isomorphic to $\wZ \widehat{C}_{*}(X)$. Nevertheless, since
the complex $\Z \widehat{S}_{*}(X)/\widehat{D}^s_*(X)$ is quasi-isomorphic
to $\Z \widehat{S}_{*}(X)$ (due to the fact that the complex of degenerate simplices of a simplicial set is acyclic), it seems
reasonable to think that there exists a complex which is
quasi-isomorphic to $\wZ \widehat{C}_{*}(X)$ and which factors the morphism $\ch$ as
above. This is done in the sequel.

Let $\mmP$ be a small exact category. The smallest complex to consider is the following. For every $n$,
let
$$C^{deg}_{n}(\mmP):=\{ \Cub(s_iE),\ E\in S_{n}(\mmP),\ i\in \{1,\dots,n-1\}\}. $$
Let $\Z C^{deg}_{n}(\mmP)$ be the free abelian group on
$C^{deg}_{n}(\mmP)$ and let
$$\wZ C^{deg}_{n}(\mmP) := \frac{\Z C^{deg}_{n}(\mmP)+\Z D_n(\mmP)}{\Z D_n(\mmP)}. $$

\begin{lmm}\label{cubsdeg} Let $E\in S_{n}(\mmP)$.
\begin{enumerate}[(i)]
\item $d\Cub(s_iE) \in \Z C^{deg}_{n-1}(\mmP)+\Z D_{n-1}(\mmP)$, for all
$i=1,\dots,n-1$.
 \item For $i=1,\dots,n-1$, the equality  $$ d\Cub(s_iE)=
\sum_{j=0}^{i-1}(-1)^{j+1}\Cub(s_{i-1}
\partial_jE) + \sum_{j=i+1}^n(-1)^{j} \Cub(s_i\partial_jE)$$
holds in  $\wZ C^{deg}_{n-1}(\mmP)$.
\end{enumerate}
\end{lmm}
\begin{proof}
By definition,
$$d\Cub(s_iE) = \sum_{j=1}^{n}\sum_{l=0}^2(-1)^{j+l}\partial_j^l\Cub(s_{i} E). $$
Since $\partial_i^l\tau_i=\partial_{i+1}^l$ for all $l=0,1,2$, by
\eqref{sifaces} we have
$$\partial_i^l\Cub(s_iE)=\partial_{i+1}^l\Cub(s_iE).$$ Hence these
two terms cancel each other in the previous sum. So, assume
that $j\neq i,i+1$. If $l=1$, then, by \eqref{facescubes},
$$\partial_j^1\Cub(s_iE)= \Cub(\partial_js_iE)=\left\{\begin{array}{ll}
\Cub(s_{i-1}\partial_jE) & j<i, \\ \Cub(s_i\partial_{j-1}E) & j>i+1.
\end{array} \right.  $$
If $l=0$ and $j\neq n$, or $l=2$ and $j\neq 1$,
$\partial_j^l\Cub(s_iE)$ is a degenerate cube and hence it is zero
in the group $\wZ C^{deg}_{n}(\mmP)$. Finally, we have
\begin{eqnarray*}
\partial_n^0 \Cub(s_iE) &=& \Cub(s_{i}\partial_nE), \\ \partial_1^2\Cub(s_iE)
&=& \Cub(s_{i-1}\partial_0E). \end{eqnarray*} The statements of
the lemma follow from these calculations and the relations \eqref{sifaces}.
\end{proof}

It follows from the last lemma that $\wZ C^{deg}_{*}(\mmP)$ is
a chain complex with the differential induced
by the differential of $\wZ C_*(\mmP)$.

\begin{prp}
The complex $\wZ C^{deg}_{*}(\mmP)$ is
quasi-isomorphic to zero.
\end{prp}
\begin{proof}
We prove the proposition by constructing a chain of chain complexes
\begin{equation}\label{chainci}
0=C^0_*\subset C^1_* \subset \cdots \subset C^{n-2}_* \subset C^{n-1}_*=\wZ C^{deg}_{*}(\mmP),
\end{equation}
such that   all the quotients $C^i_*/C^{i-1}_*$
are homotopically trivial, that is, there exists
a homotopy
$$h_n: C^i_n/C^{i-1}_n \rightarrow C^i_{n+1}/C^{i-1}_{n+1}$$
such that
$$dh_n+h_{n-1}d=id.$$
It means in particular that for every $i$, the complex
$C^i_*/C^{i-1}_*$ is quasi-isomorphic to zero. Then, since
$C_*^0=0$, it follows inductively that $C_*^i$ is quasi-isomorphic
to zero for all $i$ and the proposition is proved.

For every $i=1,\dots,n-1$, let
$$\Z C^{deg,i}_{n}(\mmP)=\{\Cub(s_jE),\ E\in S_{n}(\mmP),\ j\in \{1,\dots,i\} \},$$
and let
$$C^i_n= \frac{\Z C^{deg,i}_{n}(\mmP)+\Z D_n(\mmP)}{\Z D_n(\mmP)}. $$
By Lemma \ref{cubsdeg}, (ii), $C^i_*$ are chain
complexes with the differential induced by
the differential of $\wZ C_*(\mmP)$. Moreover, for every $i$ there is an inclusion of complexes
$C_*^i\subseteq C_*^{i+1}$.

Fix $E\in S_{n}(\mmP)$ and an index $i$. Consider an element
$\Cub(s_iE)\in C^i_*/C^{i-1}_*$ and define
$$h_n(\Cub(s_iE)) = (-1)^{i+1} \Cub(s_is_iE).$$

Then, by Lemma \ref{cubsdeg}, in the complex $C^i_*/C^{i-1}_*$,
\begin{eqnarray*}
d\Cub(s_iE) &=& \sum_{j=i+1}^{n+1}(-1)^{j} \Cub(s_i\partial_j E)
\end{eqnarray*}
and
\begin{eqnarray*}
dh_n(\Cub(s_iE)) &=& \sum_{j=i+1}^{n+1}(-1)^{i+j+1} \Cub(s_i\partial_j s_iE)
\\ &=& \Cub(s_iE)+
 \sum_{j=i+2}^{n+1}(-1)^{i+j+1} \Cub(s_is_i\partial_{j-1}E) \\
 &=& \Cub(s_iE) + h_{n-1}(d\Cub(s_iE)).
\end{eqnarray*}
Therefore, we have proved that $C^i_*/C^{i-1}_*$ is homotopically trivial.
\end{proof}

Let
$$\wZ C_{*}^s(\mmP):= \frac{\Z C_{*}(\mmP)}{\Z D_*(\mmP)+\Z C^{deg}_{*}(\mmP)}.$$

\begin{crl}\label{modifcubes}
The natural chain morphism
$$\wZ C_{*}(\mmP) \rightarrow \wZ C_{*}^s(\mmP) $$
is a quasi-isomorphism.
\end{crl}
\vist

If   $\mmP=\widehat{\mmP}(X)$, we simply write $\wZ
\widehat{C}_{*}^s(X):=\wZ C_{*}^s(\widehat{\mmP}(X))$ and
 $\wZ \widehat{C}^{deg}_{*}(X):=\wZ \widehat{C}^{deg}_{*}(\widehat{\mmP}(X))$. Since ``$\ch$''
is zero on $\Z \widehat{D}_*(X)+\Z \widehat{C}^{deg}_{*}(X)$, we
have obtained the following corollary.

\begin{crl}\label{modifcubes2} Let $X$ be a proper arithmetic variety over $\Z$.
\begin{enumerate}[(i)]
\item The map $\ch$ admits a factorization as
$$C_{*}(|\widehat{S}_{\cdot}(X)|)[-1]
 \xrightarrow{\Cub} \wZ \widehat{C}_{*}^s(X) \xrightarrow{\ch} \bigoplus_{p\geq 0}\mmD^{2p-*}(X,p).$$
\item The natural morphism
$$ \wZ \widehat{C}_{*}(X) \xrightarrow{\sim} \wZ \widehat{C}_{*}^s(X)$$
is a quasi-isomorphism.
\end{enumerate}
\end{crl}
\vist

At this point, we have all the ingredients to prove that there is
an isomorphism between $\widehat{K}^{T}_n(X)_{\Q}$ and $
\widehat{H}_{n}(\wZ \widehat{C}_{*}(X),\ch)_{\Q}.
$

For the proof of next theorem recall that the Hurewicz morphism
$$\pi_n(|\widehat{S}_{\cdot}(X)|) \rightarrow H_n(|\widehat{S}_{\cdot}(X)|) $$
maps the class of a pointed map $S^n \xrightarrow{f}
|\widehat{S}_{\cdot}(X)|$ to $f_*([S^n]).$

\begin{thm}\label{arithrational} Let $X$ be a proper arithmetic variety over $\Z$. Then,
for every $n\geq 0$, there is an isomorphism
$$\widehat{K}^{T}_n(X)_{\Q} \xrightarrow{\cong}  \widehat{H}_{n}(\wZ \widehat{C}_{*}(X),\ch)_{\Q}. $$
Moreover, there are commutative diagrams {\small
$$\begin{array}{cc}\xymatrix@C=15pt{ \widehat{K}^{T}_n(X)_{\Q}
\ar[r]^{\zeta} \ar[d]_{\cong} & K_n(X)_{\Q} \ar[d]^{\cong}
\\ \widehat{H}_{n}(\wZ \widehat{C}_{*}(X),\ch)_{\Q} \ar[r]^{\zeta} & H_n(\wZ
\widehat{C}_*(X),\Q) } & \xymatrix@C=15pt{
\widehat{K}^{T}_n(X)_{\Q} \ar[r]^(0.35){\ch} \ar[d]_{\cong} &
\bigoplus_{p\geq 0}\mmD^{2p-*}(X,p)
  \ar[d]^{=}
\\ \widehat{H}_{n}(\wZ \widehat{C}_{*}(X),\ch)_{\Q} \ar[r]^(0.45){\ch} & \bigoplus_{p\geq
0}\mmD^{2p-*}(X,p). }\end{array}$$}
\end{thm}
\begin{proof} Consider the chain complex $\widetilde{\Z} \widehat{C}^s_*(X)=
\frac{\Z \widehat{C}_{*}(X)}{\Z \widehat{D}_*(X)+\Z
\widehat{C}^{deg}_{*}(X)}$, defined before Corollary
\ref{modifcubes2}. Let $\widehat{H}_{*}(\widetilde{\Z}
\widehat{C}^s_*(X),\ch)$ denote the modified homology groups with
respect to the morphism
$$\wZ \widehat{C}^s_{*}(X) \xrightarrow{\ch}
\bigoplus_{p\geq 0}\mmD^{2p-*}(X,p). $$ Consider the following
commutative diagram:
$$\xymatrix{
\widetilde{\Z} \widehat{C}_*(X) \ar[d]_{\sim} \ar[r]^(0.37){\ch} &  \bigoplus_{p\geq 0}\mmD^{2p-*}(X,p) \ar[d]^{=} \\
 \widetilde{\Z} \widehat{C}^s_*(X)  \ar[r]^(0.37){\ch} &  \bigoplus_{p\geq
 0}\mmD^{2p-*}(X,p).
}$$ By Lemma \ref{arithlong}, there is an induced isomorphism
$$ \widehat{H}_{n}(\wZ \widehat{C}_{*}(X),\ch) \xrightarrow{\pi}   \widehat{H}_{n}(\widetilde{\Z} \widehat{C}^s_*(X),\ch),$$
which commutes with $\zeta$.

It remains to prove that there is an isomorphism
$$\widehat{K}^{T}_n(X)_{\Q} \cong   \widehat{H}_{n}(\widetilde{\Z} \widehat{C}^s_*(X),\ch)_{\Q},$$
commuting with $\zeta$.

Consider the chain morphism
$$C_*(|\widehat{S}_{\cdot}(X)|)[-1] \xrightarrow{\Cub} \widetilde{\Z} \widehat{C}^s_*(X).$$

Recall from section \ref{higherK} that the isomorphism
$$ K_{n}(X)_{\Q}  \xrightarrow{\Cub}   H_n(\wZ \widehat{C}_*(X),\Q)$$
is given by the composition
{\small $$ K_{n}(X)=\pi_{n+1}(|\widehat{S}_{\cdot}(X)|)_{\Q} \xrightarrow{ \mathrm{Hurewicz}}
H_n(C_{*}(|\widehat{S}_{\cdot}(X)|)[-1])_{\Q} \xrightarrow{\Cub}
H_n(\wZ \widehat{C}_*(X),\Q)$$} which sends the class of a
cellular map $[f:S^{n+1}\rightarrow |\widehat{S}_{\cdot}(X)|]$ to
$\Cub f_*([S^{n+1}])$.

If $f,f':S^{n+1}\rightarrow |\widehat{S}_{\cdot}(X)|$ are
homotopic with cellular homotopy $h$, then
$$dh_*[S^{n+1}\times I]=(-1)^{n+1}(f'_*[S^{n+1}]-f_*[S^{n+1}])$$
in $C_{*}(|\widehat{S}_{\cdot}(X)|)[-1]$.

Let
\begin{eqnarray*}
 \widehat{K}^{T}_{n}(X)_{\Q} & \xrightarrow{ \Cub\nolimits^s} &  \widehat{H}_n(\wZ \widehat{C}_*^s(X),\ch )_{\Q}
 \\ \
 [(f:S^{n+1}\rightarrow |\widehat{S}_{\cdot}(X)|,\omega)] & \mapsto & [(\Cub
 f_*([S^{n+1}]),-\omega)].
\end{eqnarray*}
This morphism is well defined. Indeed, let $h$ be a cellular
homotopy between two pairs $(f,\omega)$ and $(f',\omega')$. Then,
if we denote $\alpha=(-1)^{n+1}\Cub h_*([S^{n+1}\times I])$, we
have
$$d (\alpha)= \Cub f'_*([S^{n+1}])- \Cub f_*([S^{n+1}]),$$
and
$$  \ch (\alpha)=\omega-\omega'.$$

Finally, consider the diagram {\small $$\xymatrix@C=13pt{
K_{n+1}(X)_{\Q} \ar[r]^(0.45){\ch} \ar[d]_{\cong} &
\wmD^{2p-n-1}(X,p)
 \ar[r]^(0.6){a} \ar[d]_{=} &
\widehat{K}^{T}_{n}(X)_{\Q} \ar[r]^{\zeta} \ar[d]_{ \Cub\nolimits^s} &  K_{n}(X)_{\Q} \ar[d]_{\cong} \ar[r] & 0    \\
H_{n+1}(\widetilde{\Z} \widehat{C}^s_*(X))_{\Q} \ar[r]^(0.50){\ch}
& \wmD^{2p-n-1}(X,p) \ar[r]^(0.45){a} &  \widehat{H}_{n}(\wZ
\widehat{C}^s_*(X),\ch )_{\Q} \ar[r]^(0.55){\zeta} &
H_{n}(\widetilde{\Z} \widehat{C}^s_*(X))_{\Q}  \ar[r] & 0 }
$$}
Since the rows are exact sequences, the statement of the
proposition  follows from the five lemma.
\end{proof}

\begin{crl}\label{normisohermi}
Let $X$ be a proper arithmetic variety over $\Z$. Then, for every
$n\geq 0$, there is an isomorphism
$$\widehat{K}^{T}_n(X)_{\Q} \xrightarrow{\cong}  \widehat{H}_{n}(N\widehat{C}_{*}(X),\ch)_{\Q}. $$
\end{crl}
\vist

\medskip
\textbf{Product structure on rational arithmetic $K$-theory.}
Takeda, in \cite{Takeda}, defines a product structure for
$\widehat{K}_n^T(X)$ compatible with the Loday product of
algebraic $K$-theory, and for which the morphism
$$\ch:
\widehat{K}_n^T(X) \rightarrow \bigoplus_{p\geq 0}\mmD^{2p-*}(X,p)$$ is a ring
morphism (see  loc. cit., Proposition 6.8). Since there is a
natural isomorphism
$$\widehat{K}_n(X)\cong \ker(\ch:
\widehat{K}_n^T(X)\rightarrow \bigoplus_{p\geq 0}\mmD^{2p-n}(X,p)),$$ there is an
induced Loday product on $\widehat{K}_n(X)$.

In algebraic $K$-theory, the Adams operations are derived from the
lambda operations by a polynomial relation. In order to do that,
the product structure for $\bigoplus_{n\geq 0}K_n(X)$ is the one
for which $\bigoplus_{n\geq 1}K_n(X)$ is a square zero ideal.

Therefore, we consider the product structure on $\bigoplus_{n\geq
0}\widehat{K}_n(X)$ for which $\bigoplus_{n\geq
1}\widehat{K}_n(X)$ is a square zero ideal and agrees with the
Loday product otherwise.

After tensoralizing with $\Q$, and using the description of
$\widehat{K}_n(X)_{\Q}$ via the isomorphism
$$\widehat{K}_n(X)_{\Q} \cong
H_n(s(\widehat{\ch}),\Q),$$ the product is defined as follows.

\begin{lmm}\label{productarithktheory}
Let $(\overline{E},\alpha)\in \widehat{K}_0(X)_{\Q}$ and let
$(\overline{F},\beta)\in \widehat{K}_n(X)_{\Q}$ with $n\geq 0$.
Then, by the product structure on $\widehat{K}_n(X)_{\Q}$ induced
by the product structure on $\widehat{K}_n^T(X)_{\Q}$, we have
$$(\overline{E},\alpha)\otimes (\overline{F},\beta)=
(\overline{E}\otimes \overline{F},\alpha \bullet
\ch(\overline{F})+\ch(\overline{E})\bullet \beta - \alpha\bullet
d_{\mmD}(\beta)) \in \widehat{K}_n(X)_{\Q}.$$
\end{lmm}
\vist

\begin{rem}
With the notation of the previous lemma, if $n>0$ then
$\overline{F}$ is an $n$-cube such that $d\overline{F}=0$ and
$\beta\in\bigoplus_{p\geq 0} \mmD^{2p-n-1}(X,p)$ is a differential
form such that $\ch(\overline{F})=d_{\mmD}(\beta)$. Hence,
$$\alpha \bullet \ch(\overline{F})=\alpha\bullet d_{\mmD}(\beta) $$
and therefore,
$$(\overline{E},\alpha)\otimes (\overline{F},\beta)=
(\overline{E}\otimes \overline{F},\ch(\overline{E})\bullet
\beta).$$
\end{rem}

\section{Adams operations on higher arithmetic $K$-theory}

In this section we construct the Adams operations on the higher
arithmetic $K$-groups tensored by the rational numbers. The
construction is adapted to both definitions of higher arithmetic
$K$-groups.

Let $X$ be a proper arithmetic variety over $\Z$.
In \cite{FeliuAdams}, we defined a chain morphism inducing Adams operations on
higher algebraic $K$-theory (with rational coefficients), using the chain complex of cubes. The results stated in \cite{FeliuAdams} for the category of locally free sheaves of finite rank over $X$, $\mmP(X)$, translate to the category $\widehat{\mmP}(X)$, considering the appropriate metrics.
As it will be overviewed next, the definition of the algebraic Adams operations involves the Koszul complex of a locally free sheaf and the isomorphism \eqref{koszulsum}. Therefore, the metrics of the Koszul complex are imposed by the need of the equivalent isometry \eqref{koszulisosum} to hold.

In order to define Adams operations on the  higher arithmetic $K$-groups tensored by $\Q$, it would be desirable to have the Chern character morphism commute with the Adams operations on the complex of cubes (diagram \eqref{adamsarith}). To this end, the Bott-Chern form of the Koszul complex should be zero. However, with the hermitian metric on the Koszul complex imposed by \eqref{koszulisosum}, its Bott-Chern form does not vanish.
This problem can be solved by introducing a slight modification to the definition of Adams operations of \cite{FeliuAdams} (see Remark \ref{modadams}).

We start this section with an overview of the Adams operations defined in \cite{FeliuAdams}, together with the required modification. In the next subsection we proceed to discuss the Bott-Chern form of the Koszul complex. We finish this section by showing the commutativity of diagram \eqref{adamsarith} and by deducing the Adams operations on higher arithmetic $K$-theory tensored by $\Q$.

\subsection{Adams operations on higher algebraic $K$-theory}
We recall here briefly the key points of the definition of Adams operations of \cite{FeliuAdams}, with a slight modification introduced.
In the following definitions, $\cong$ denotes an isometry.

\paragraph{The Koszul complex.}
For every locally free sheaf $E$ of finite rank on any variety, the $k$-th Koszul
complex of $E$ is the exact sequence
$$\Psi^k(E)^*: 0\rightarrow \Psi^k(E)^0 \xrightarrow{\varphi_0} \dots\xrightarrow{\varphi_{k-1}} \Psi^k(E)^k \rightarrow 0$$
with
$$\Psi^k(E)^p= E{\cdot}\stackrel{p}{\dots}{\cdot} E \otimes E\wedge \stackrel{k-p}{\dots}
\wedge E= S^pE \otimes \bigwedge\nolimits^{k-p}E.$$
The arrows $\varphi_p$ are defined as follows. Consider the inclusions
$$ S^pE  \xrightarrow{\iota_p}  T^pE,\quad\textrm{and }\quad \bigwedge\nolimits^pE  \xrightarrow{j_p}  T^pE
$$ defined locally by
\begin{eqnarray}
 \iota_p(x_{i_1}{\cdot} \dots {\cdot} x_{i_p})&=&  
   \sum_{\sigma\in
\mathfrak{S}_p} x_{\sigma(i_1)} \otimes \ldots \otimes x_{\sigma(i_p)},  \label{iota}\\
j_p(x_{i_1}\wedge \ldots \wedge x_{i_p}) & = & %
\sum_{\tau\in \mathfrak{S}_p}
(-1)^{|\tau|}\ x_{\tau(i_1)} \otimes \ldots \otimes x_{\tau(i_p)}. \nonumber
\end{eqnarray}
Consider the natural projections
$$\begin{array}{rclcrcl}
 T^pE &  \xrightarrow{\pi_p} & S^pE, &  \textrm{and}  &
 T^pE & \xrightarrow{\rho_p}&
\bigwedge\nolimits^pE
\\ x_{i_1}\otimes  \ldots \otimes x_{i_p} & \mapsto & x_{i_1}{\cdot} \ldots {\cdot}
x_{i_p}, & \    &
 x_{i_1}\otimes  \ldots \otimes x_{i_p} & \mapsto & x_{i_1}\wedge \ldots \wedge
x_{i_p}.
\end{array}$$
For every $p$, the  morphisms
$$\varphi_p: S^pE  \otimes \bigwedge\nolimits^{k-p}E\rightarrow
S^{p+1}E \otimes
\bigwedge\nolimits^{k-p-1}E $$ in the Koszul complex are
given as
\begin{equation}\label{varphi}\varphi_p = \frac{1}{p!(k-p-1)!} (\pi_{p+1}\otimes \rho_{k-p-1})\circ (\iota_p\otimes j_{k-p})\end{equation}
(see \eqref{varphidef} for the explicit computation of $\varphi_p$).

The key properties of the Koszul complex that make it suitable for the definition of Adams operations on higher algebraic $K$-theory are the following:
\begin{enumerate*}[$\blacktriangleright$]
\item If $E$ and $F$ are two locally free sheaves of finite rank, there is a canonical isomorphism of complexes
\begin{equation}\label{koszulsum}
\Psi^k(E\oplus F)^* \cong \bigoplus_{p=0}^k \Psi^p(E)^*\otimes \Psi^{k-p}(F)^*.\end{equation}
\item The secondary Euler characteristic class  of the Koszul complex
\begin{equation} \Psi^k(E)=  \sum_{p\geq 0}(-1)^{k-p+1} (k-p) \Psi^k(E)^p \end{equation}
agrees with the $k$-th Adams operation of $E$ in $K_0(X)$.
\end{enumerate*}

If $\overline{E}$ is a hermitian vector bundle, there is a naturally induced metric on the tensor product
$T^k\overline{E}=\overline{E}\otimes\stackrel{k}{\dots}\otimes
\overline{E}$. Endow $\bigwedge\nolimits^k\overline{E}$ with the wedge product metric, that is, the metric induced by the natural inclusion
of $\bigwedge\nolimits^k\overline{E}$ into $T^k\overline{E}$ given by $\frac{1}{\sqrt{k!}} j_k$. Analogously, we endow $S^k\overline{E}$ with the hermitian metric induced by the natural inclusion of $S^k\overline{E}$ into $T^k\overline{E}$ given by $\frac{1}{\sqrt{k!}} \iota_k$.

In this way, if $e_{1},\dots,e_{n}$ is a orthonormal local frame in $\overline{E}$, then, the set $\{e_{i_{1}}\otimes \ldots \otimes e_{i_{k}}\}_{i_{1},\dots,i_{k}\in \{1,\dots,n\}}$ forms a orthonormal basis of $T^k(\overline{E})$ and the set $\{e_{i_{1}}\wedge \ldots \wedge e_{i_{k}}\}_{i_{1}<\dots<i_{k}\in \{1,\dots,n\}}$ forms a orthonormal basis of $\bigwedge^k(\overline{E})$. The set $\{e_{i_{1}}\cdot \ldots \cdot e_{i_{k}}\}_{i_{1}\leq \dots \leq i_{k}\in \{1,\dots,n\}}$ forms a orthogonal basis of $S^k(\overline{E})$ with the norm of each element depending on the number of repetitions among the subindices. In particular, if $i_{1}<\dots<i_{k}$, then, the norm of $e_{i_{1}}\cdot \ldots \cdot e_{i_{k}}$ is one.

Denote these metrics by $h_{S^p(\overline{E})},h_{\bigwedge\nolimits^p(\overline{E})},h_{T^p(\overline{E})}$.
The locally free sheaves $\Psi^k(\overline{E})^p$ are endowed with the tensor product metric.
With these hermitian metrics,  if $\overline{E}$ and $\overline{F}$ are hermitian vector bundles, the algebraic canonical isomorphism of complexes \eqref{koszulsum} is an isometry of hermitian complexes
\begin{equation}\label{koszulisosum}\Psi^k(\overline{E}\oplus \overline{F})^* \cong \bigoplus_{p=0}^k \Psi^p(\overline{E})^*\otimes \Psi^{k-p}(\overline{F})^*.\end{equation}
Indeed, the natural inclusions
$$S^r\overline{E} \otimes \bigwedge\nolimits^{p-r} \overline{E} \otimes S^l\overline{F} \otimes \bigwedge\nolimits^{k-p-l} \overline{F} \hookrightarrow
S^{r+l}(\overline{E}\oplus \overline{F}) \otimes \bigwedge\nolimits^{k-r-l}(\overline{E}\oplus \overline{F})   $$
are compatible with the defined hermitian metrics.

\paragraph{Hermitian split cubes.}
\begin{dfn}[(cf. \cite{FeliuAdams})] Let $X$ be a proper arithmetic variety.
Let $\{ \overline{E}^{\bj}\}_{\bj\in \{0,2\}^n}$ be a collection of hermitian vector bundles over $X$, indexed by $\{0,2\}^n$. Let $[\overline{E}^{\bj}]_{\bj\in \{0,2\}^n}$ be the hermitian
$n$-cube defined  as follows.
\begin{enumerate*}[$\rhd$]
\item Let $\bj\in \{0,1,2\}^n$ and let $u_1<\dots < u_s$ be the indices with $j_{u_i}=1$. We denote $(v_1,\dots,v_n)=\sigma_{m_1,\dots,m_s}(\bj)$ to be the multi-index with
$$v_k =\left\{ \begin{array}{ll}
j_k & \textrm{ if }k\neq u_l,\textrm{ for all }l, \\
m_l & \textrm{ if }k=u_l. \end{array}\right.
$$
Then, the $\bj$-component of  $[\overline{E}^{\bj}]_{\bj\in \{0,2\}^n}$ is
$$\bigoplus_{(m_1,\dots,m_s)\in\{0,2\}^{s} }^{\bot} \overline{E}^{
\sigma_{m_1,\dots,m_s}(\bj)}. $$
\item
The morphisms are compositions of the following canonical morphisms:
$$\begin{array}{rclcrcl}
A \oplus^{\bot} B & \twoheadrightarrow & A, & \qquad &  A\oplus^{\bot} B & \xrightarrow{\cong} & B \oplus^{\bot} A,\\
A & \hookrightarrow &  A \oplus^{\bot} B,  & \qquad & A \oplus^{\bot} (B\oplus^{\bot} C) & \xrightarrow{\cong}
 & (A\oplus^{\bot} B) \oplus^{\bot} C.
\end{array}$$
\end{enumerate*}
 \nopagebreak[0]
 A hermitian $n$-cube of this form is called a \emph{direct sum hermitian $n$-cube}.
\end{dfn}

\begin{dfn}[(cf. \cite{FeliuAdams})]\label{splitcube} Let $X$ be a proper arithmetic variety.
\begin{enumerate*}[$\blacktriangleright$] \item Let $\overline{E}$ be a hermitian $n$-cube. The \emph{direct sum hermitian $n$-cube associated
to $\overline{E}$}, $\Sp(\overline{E})$, is the hermitian $n$-cube
$$\Sp(\overline{E}):=[\overline{E}^{\bj}]_{\bj\in \{0,2\}^n}.$$  \item A \emph{ hermitian split $n$-cube} is
a couple $(\overline{E},f)$, where $\overline{E}$ is a hermitian $n$-cube and $f:\Sp(\overline{E})\rightarrow \overline{E}$ is an
isometry of hermitian $n$-cubes such that $f^{\bj}=id$ if $\bj\in \{0,2\}^{n}$.
The morphism $f$ is called the \emph{splitting} of $(\overline{E},f)$.
\end{enumerate*}
\end{dfn}

Roughly speaking, these are the cubes which are orthogonal direct sums in all
directions. Let
$$\Z \widehat{\Spn}_n(X):= \Z \{\textrm{split } \textrm{hermitian n-cubes in } X \}$$ and let $\Z \widehat{\Spn}_*(X)=\bigoplus_n \Z \widehat{\Spn}_n(X)$.
As shown in \cite{FeliuAdams}, there is a differential map
 $$d :   \Z \widehat{\Spn}_n(X) \rightarrow \Z \widehat{\Spn}_{n-1}(X)$$
 making $(\Z \widehat{\Spn}_*(X),d)$ a chain complex and in such a way that the morphism that forgets the splitting
$\Z \widehat{\Spn}_*(X) \rightarrow \Z \widehat{C}_*(X)$  is a chain morphism.

\paragraph{Some intermediate chain morphisms.}
The construction of Adams operations factors through an intermediate complex. We recall here its construction due to the fact that a slight modification
needs to be introduced.

Let $k\geq 1$. For every $n\geq 0$ and $i=1,\dots,k-1$, we define
\begin{eqnarray*}
\widehat{G}^k_1(X)_n & := & \{\textrm{acyclic cochain
complexes of length }k\textrm{ of hermitian }n\textrm{-cubes} \}, \\
\widehat{G}^{i,k}_2(X)_n &:= &
\{\textrm{2-iterated acyclic cochain complexes of lengths }(k-i,i) \\ && \
\textrm{ of hermitian }n\textrm{-cubes}
 \}.
\end{eqnarray*}  The differential of $\Z \widehat{C}_*(X)$ induces a
differential on the graded abelian groups
$$\Z \widehat{G}^{i,k}_2(X)_*:=\bigoplus_n \Z \widehat{G}^{i,k}_2(X)_n\qquad  \textrm{and}\qquad \Z \widehat{G}^k_1(X)_*:=\bigoplus_n
\Z \widehat{G}^k_1(X)_n.$$

For every $n$, the simple complex associated to a $2$-iterated cochain complex induces a morphism
$$\Phi^i:\Z \widehat{G}^{i,k}_2(X)_n\rightarrow \Z \widehat{G}^k_1(X)_n. $$
A new chain complex is defined by setting
$$\Z \widehat{G}^{k}(X)_n := \bigoplus_{i=1}^{k-1}\Z \widehat{G}^{i,k}_2(X)_{n-1}  \oplus \Z \widehat{G}^k_1(X)_n. $$
If $\overline{B}_i\in \widehat{G}^{i,k}_2(X)_{n-1}$, for $i=1,\dots,k-1$, and $\overline{A}\in \widehat{G}^k_1(X)_n$, the differential is given by
$$d_s(\overline{B}_1,\dots,\overline{B}_{k-1},\overline{A}):= (-d\overline{B}_1,\dots,-d\overline{B}_{k-1}, \sum_{i=1}^{k-1}(-1)^{i}\Phi^i(\overline{B}_i)+d\overline{A}). $$

\begin{dfn}
For any acyclic cochain complex of hermitian $n$-cubes
$
\overline{A}: 0 \rightarrow \overline{A}^0 \xrightarrow{f^0} \cdots \xrightarrow{f^{j-1}} \overline{A}^j  \xrightarrow{f^j}
\cdots \xrightarrow{f^{k-1}} \overline{A}^k \rightarrow 0,
$
we define:
\begin{enumerate*}[$\rhd$]
\item  $\widehat{\varphi}_1(\overline{A})$ to be the \emph{secondary Euler characteristic class}, i.e.
$$\widehat{\varphi}_1(\overline{A}) = \sum_{p\geq 0}(-1)^{k-p+1} (k-p) \overline{A}^p \in \Z \widehat{C}_n(X).$$
\item  $\mu(\overline{A}):= \sum_{j\geq 0} (-1)^{j-1}\mu^j(\overline{A})$ where $\mu^j(\overline{A})$ is the hermitian $n$-cube defined by
$$\partial_1^0 (\mu^j(\overline{A}))=\ker f^j, \quad \partial_1^1(\mu^j(\overline{A}))=\overline{A}^j, \quad \textrm{and}\quad
\partial_1^2(\mu^j(\overline{A}))=\ker f^{j+1}.$$
\item $\lambda_k(\overline{A})$ is the  acyclic cochain complex of hermitian $n$-cubes
$$0 \rightarrow \overline{A}^0 \xrightarrow{\frac{1}{\sqrt{k}}f^0} \cdots \xrightarrow{\frac{1}{\sqrt{k}}f^{j-1}} \overline{A}^j  \xrightarrow{\frac{1}{\sqrt{k}}f^j}
\cdots \xrightarrow{\frac{1}{\sqrt{k}}f^{k-1}} \overline{A}^k \rightarrow 0.$$
\end{enumerate*}
\end{dfn}

If $\overline{B}_i \in \widehat{G}_2^{i,k}(X)_n$, then $\overline{B}_i$ is a 2-iterated acyclic cochain complex
where $\overline{B}_i^{j_1j_2}$ is a hermitian $n$-cube, for every $j_1,j_2$.  We attach to it a sum
of exact sequences of hermitian $n$-cubes as follows:
\begin{eqnarray*}
\widehat{\varphi}_2(\overline{B}_i) & = & \sum_{j\geq
0}(-1)^{k-j+1}((k-i-j)\lambda_{k-i}(\overline{B}_i^{*,j})+(i-j)\lambda_{i}(\overline{B}_i^{j,*}))
\\ &&+ \sum_{s\geq 1}(-1)^{k-s} (k-s) \sum_{j\geq 0} (\overline{B}_i^{s-j,j}\rightarrow
\bigoplus_{j'\geq j} \overline{B}_i^{s-j',j'}\rightarrow
\bigoplus_{j'>j}\overline{B}_i^{s-j',j'}).\end{eqnarray*} Roughly speaking, the first
summand corresponds to the secondary Euler characteristic of the rows and the
columns. The second summand appears as a correction factor for the fact that
direct sums are not sums in $\Z \widehat{C}_n(X)$.

\begin{lmm}
The morphism given for every $n$ by
\begin{eqnarray}
\Z \widehat{G}^k(X)_{n} & \xrightarrow{\varphi} & \Z \widehat{C}_n (X) \label{morphismvarphi} \\
(\overline{B}_1,\dots,\overline{B}_{k-1},\overline{A}) &\mapsto &
\widehat{\varphi}_1(\overline{A})+\sum_{i=1}^{k-1}(-1)^{i+1}\mu(\widehat{\varphi}_2(\overline{B}_i)).\nonumber
\end{eqnarray}
is a chain morphism.
\end{lmm}
\begin{proof} See \cite{FeliuAdams}. It follows essentially from the three equalities
\begin{eqnarray*}
d\widehat{\varphi}_1(\overline{A})&=& \widehat{\varphi}_1(d\overline{A}), \\
-\widehat{\varphi}_2(d\overline{B}_i) & = & \sum_{l=2}^n\sum_{j=0}^2 (-1)^{l+j} \partial_l^j
\widehat{\varphi}_2(\overline{B}_i),\quad \forall i, \\
 \widehat{\varphi}_1(\Phi^i(\overline{B}_i)) & =& \sum_{r\geq 0}(-1)^r \partial_1^r \widehat{\varphi}_2(\overline{B}_i)\quad \forall i.
\end{eqnarray*}
\end{proof}

\begin{rem}\label{modadams}
The definition given here for $\widehat{\varphi}_2$ is a slight modification of the definition given in the original paper  \cite{FeliuAdams}. It consists of  the twist by $1/\sqrt{k}$ of $\lambda_k$. The modification has been introduced in order to have a representative of the Adams operations on higher algebraic $K$-theory that commutes strictly with the Chern character (the reason will become apparent in the next section (cf. Proposition \ref{koszulvanish})). This modification does not alter the fact that the final morphism represents the Adams operations on higher rational algebraic $K$-theory, since the morphism at $n=0$ remains unchanged.
\end{rem}

\paragraph{Adams operations.}
In \cite{FeliuAdams}, Adams operations are constructed for every split (hermitian)
$n$-cube, using the secondary Euler characteristic class of the Koszul complex.
The following proposition follows from the construction of the Adams operations in \cite{FeliuAdams}.

\begin{prp}[(cf. \cite{FeliuAdams})]\label{adams2}  Let $X$ be a proper arithmetic variety. For every $k\geq 0$, there is a chain complex
$$\Psi^k : \Z \widehat{\Spn}_*(X) \rightarrow \wZ \widehat{C}_*(X).$$
For every $\overline{E} \in \Z \widehat{\Spn}_n(X)$, $\Psi^k(\overline{E})$ consists of a sum of hermitian
$n$-cubes of the following form:
\begin{enumerate}[(i)]
\item If $\overline{E}$ is a hermitian vector bundle (i.e. $n=0$), then
$\Psi^k( \overline{E})$
is the secondary Euler characteristic class of the Koszul complex.
\item If $n\geq 2$, the image of $\Psi^k$ consists of hermitian $n$-cubes which are split in at least one direction.
\item For $n= 1$, there are two types of summands. Some of the terms of $\Psi^k(\overline{E})$ are hermitian split $1$-cubes. The rest
are hermitian $1$-cubes of the form $\mu(\overline{A})$ with $\overline{A}= \lambda_l(\Psi^{l}(\overline{F})^*)\otimes \overline{G}$ or $\overline{A}= \overline{F}\otimes \lambda_l(\Psi^l(\overline{G})^*)$ for some hermitian bundles $\overline{F},\overline{G}$ in the entries of $\overline{E}$ and $0\leq l \leq k$. These terms arise as the image by $\mu\circ\widehat{\varphi}_2$ of $2$-iterated   acyclic cochain complexes of lengths $(k-i,i)$ of the form
$\Psi^{k-i}(\overline{G})^* \otimes \Psi^{i}(\overline{F})^*$.
\end{enumerate}
\end{prp}
\vist

\begin{rem}
  Item $(ii)$ of last proposition follows from the construction of the Adams operations in \cite{FeliuAdams}, using the isometry \eqref{koszulisosum}. These hermitian cubes appear by the image by $\varphi$ of some elements  $\widehat{A}\in \widehat{G}_1^{k}(X)_n$, $\widehat{B_i} \in \widehat{G}_2^{i,k}(X)_{n-1}$, which consist of (2-iterated) acyclic cochain complexes of hermitian $n$-cubes or $(n-1)$-cubes which are hermitian split in all directions. Note that the statement is true with the modification by $\lambda_k$ introduced here since the multiplication by the constants is not performed in the hermitian split directions.
\end{rem}

\paragraph{The transgression morphism.}

Once the Adams operations are defined for all split hermitian cubes, the final
construction makes use of the transgression bundles introduced above. This allows us to assign to every hermitian $n$-cube a collection of hermitian split cubes in $X\times (\P^1)^*$.

Let $X$ be a proper arithmetic variety.
Let $X\times (\P^1)^n$   denote $X\times _{\Z}(\P^1)^n$. For $i=1,\dots,n$ and
$j=0,1$, consider the chain morphisms induced on the complex of hermitian cubes
\begin{eqnarray*}
\delta^j_i=(Id\times \delta_j^i)^* & : & \Z \widehat{C}_*(X\times (\P^1)^n)  \rightarrow
\Z\widehat{C}_*(X\times (\P^1)^{n-1}),
\\ \sigma_i=(Id\times
\sigma^i)^* & : &  \Z\widehat{C}_*(X\times (\P^1)^{n-1})  \rightarrow  \Z\widehat{C}_*(X\times
(\P^1)^{n}).
\end{eqnarray*}

As before, let $p_1,\dots,p_n$ be the projections onto the $i$-th coordinate of
$(\P^1)^n$.
Let  $\Z \widehat{C}^{\P}_{*,*}(X)$ be the 2-iterated chain complex given by
$$\Z \widehat{C}^{\P}_{r,n}(X): = \Z \widehat{C}_{r}(X\times (\P^1)^{n}),$$ and differentials
$(d,\delta)$, with $d$ the differential of the complex of cubes and $\delta = \sum (-1)^{i+j}\delta^j_i$.
Denote by $(\Z \widehat{C}^{\P}_{*}(X),d_s)$ the associated simple complex.

Let
\begin{eqnarray*}
\Z \widehat{C}^{\P}_{r,n}(X)_{deg} & = & \sum_{i=1}^n  \sigma_i(\Z
\widehat{C}^{\P}_{r,n-1}(X)) + p_i^*\mathcal{O}(1) \otimes \sigma_i(\Z
\widehat{C}^{\P}_{r,n-1}(X)), \\
\wZ \widehat{C}^{\P}_{r,n}(X)_{\deg} & = & \Z \widehat{C}^{\P}_{r,n}(X)_{deg}/ \Z
\widehat{D}_{r}(X\times (\P^1)^n)_{deg},
\end{eqnarray*} and
let
$$\wZ \widehat{C}^{\widetilde{\P}}_{r,n}(X) :=   \wZ \widehat{C}^{\P}_{r,n}(X)/\wZ \widehat{C}^{\P}_{r,n}(X)_{\deg}.$$
Denote by $(\wZ \widehat{C}^{\widetilde{\P}}_*(X),d_s)$ the simple complex
associated to this 2-iterated chain complex.

\begin{prp}[(\cite{FeliuAdams}, Proposition 3.2)]\label{dold} If $X$ is a regular noetherian scheme,
the natural morphism of complexes
$$\wZ \widehat{C}_*(X)= \wZ \widehat{C}^{\widetilde{\P}}_{*,0}(X) \rightarrow \wZ \wC^{\widetilde{\P}}_*(X)$$
induces an isomorphism on homology with coefficients in $\Q$.
\end{prp}
\vist

In \cite[$\S$3]{FeliuAdams}, a morphism
\begin{equation}\label{T}
N \widehat{C}_*(X) \xrightarrow{T} \Z \widehat{C}^{\P}_*(X)
\end{equation}
is constructed. This morphism extends the map that assigns to every hermitian  $n$-cube $\overline{E}$ its transgression $\tr_n(\overline{E})$.
Indeed, the component of $T(\overline{E})$ in $\Z \widehat{C}^{\P}_{0,n}(X)$ is $\tr_n(\lambda(\overline{E}))$. Each of the components of $T(\overline{E})$ in
$\Z \widehat{C}^{\P}_{n-i,i}(X)$, for $i>0$, consists of a linear combination of split hermitian cubes, and hence the construction $\Psi^k$ outlined above can be applied to each of these terms. The morphism $T$ maps every hermitian cube $\overline{E}$ to a collection of hermitian cubes with canonical kernels $\lambda(\overline{E})$ and then applies the transgression construction.

In this way, one obtains for every $k\geq 0$ a chain morphism
$$ \Psi^k: N \widehat{C}_*(X) \xrightarrow{\Psi^k\circ T} \wZ \widehat{C}_*^{\widetilde{\P}}(X).$$

\begin{prp}[(\cite{FeliuAdams}, Theorem 4.2)]
The Adams operations on the hi\-gher al\-gebraic $K$-groups of $X$, after tensoring by $\Q$ (as given by Gillet and Soul\'e in \cite{GilletSouleFiltrations} or Grayson in \cite{Grayson1}), are represented
by the chain morphism
$$ \Psi^k: N \widehat{C}_*(X) \xrightarrow{\Psi^k\circ T} \wZ \widehat{C}_*^{\widetilde{\P}}(X).$$
\end{prp}
\vist

\begin{rem}\label{trans2}
The fact that the image of $T$ consists of split cubes, is proved in
\cite[Lemma 3.15]{FeliuAdams}. The fact that the image consists indeed of hermitian split
cubes, follows from the fact that
\begin{eqnarray}
 \tr\nolimits_n(\overline{E})|_{y_i=0} & \cong &  \tr\nolimits_{n-1}(\partial_i^0\overline{E})
 \oplus^{\bot} \tr\nolimits_{n-1}(\partial_i^2\overline{E}),
\end{eqnarray}
for every hermitian $n$-cube $\overline{E}$ with canonical kernels, and with $\cong$ being an isometry.
\end{rem}

\subsection{The Koszul complex and Bott-Chern forms}\label{bottchernkoszul}
In this section we discuss the Bott-Chern form of the Koszul complex, by determining hermitian metrics on it that would make its Bott-Chern form to vanish. This is the cause of the modification $\lambda_k$  introduced in the definition of the algebraic Adams operations. We next perform a direct comparison of the definition of the Adams operations on locally free sheaves and the secondary Euler characteristic class of the Koszul complex.

Although it is not required for the development of this paper, we will deduce the value of the Bott-Chern form of the Koszul complex, with the hermitian metrics fixed at the beginning of this section. This will be an easy consequence of all the computations of the first part of this subsection.

The following is a known result.
\begin{lmm}\label{koszul} Let $E$ be a locally free sheaf of finite rank on any variety.
For all $k\geq 1$, the $k$-th Koszul complex of
$E$ is split, i.e.  for all $0\leq j \leq k-1$, $\mu^j(\Psi^k(E)^*)
$ is a split short exact sequence.
\end{lmm}
\begin{proof} Recall that the morphisms in the Koszul complex, $\varphi_p$, where defined as
$$\varphi_p = \frac{1}{p!(k-p-1)!} (\pi_{p+1}\otimes \rho_{k-p-1})\circ (\iota_p\otimes j_{k-p}).$$

 Let
$$\psi_p: S^{p+1}E \otimes \bigwedge\nolimits^{k-p-1}E \longrightarrow S^pE  \otimes
\bigwedge\nolimits^{k-p}E
$$
be given as
$$\psi_{p}=\frac{1}{k\ p!\ (k-p-1)!}(\pi_{p}\otimes \rho_{k-p})\circ(\iota_{p+1}\otimes j_{k-p-1}).$$
If $\psi_{p}$ is a section of $\varphi_{p}$ over $\im \varphi_{p}$, then the short exact sequence
$$0\rightarrow \ker \varphi_{p} \rightarrow  S^pE  \otimes \bigwedge\nolimits^{k-p}E \rightarrow \im \varphi_{p} \rightarrow 0$$
is split for all $p$. That is, there
is an isomorphism
$$S^pE \otimes \bigwedge\nolimits^{k-p} E \cong \ker \varphi_p \oplus \im \varphi_p. $$

In order to see that $\psi_{p}$ is a section of $\varphi_{p}$ over $\im \varphi_{p}$, we have to see that for all $e\in
S^pE  \otimes \bigwedge\nolimits^{k-p}E$, we have
$$\varphi_p\psi_p\varphi_p(e)=\varphi_p(e).$$

Assume that the rank of $E$ is $n$ and consider a local frame in $E$, $\{e_{1},\dots,e_{n}\}$.
Renaming the indices, it is only necessary to check the previous equality for an element of the form
$e= e_{i_{1}}\cdot \ldots \cdot e_{i_{p}} \otimes e_{1} \wedge \ldots \wedge e_{k-p}$
where $i_{1},\dots,i_{p}\in \{1,\dots,n\}$.
By definition,
\begin{eqnarray*}
\varphi_p(e) & = & \frac{1}{p!(k-p-1)!} \sum_{\substack{\sigma\in \mathfrak{S}_p \\ \tau\in \mathfrak{S}_{k-p}}}(-1)^{|\tau|}
e_{i_{\sigma(1)}} {\cdot} \ldots {\cdot} e_{i_{\sigma(p)}} {\cdot} e_{\tau(1)}
\otimes e_{\tau(2)} \wedge   \ldots \wedge e_{\tau(k-p)} \\
&=& \frac{1}{p!(k-p-1)!} \sum_{\tau\in \mathfrak{S}_{k-p}} (-1)^{|\tau|}p!\ e_{i_1}
{\cdot} \ldots {\cdot} e_{i_p} {\cdot} e_{\tau(1)} \otimes e_{\tau(2)}
\wedge \ldots \wedge e_{\tau(k-p)}.
\end{eqnarray*}
Observe that
if $\tau(1)=j$, then there is a decomposition $\tau=\tau'\rho$ with $\tau',\rho\in
\mathfrak{S}_p$,
$\rho(1,\ldots,k-p)=(j,1,\ldots,\widehat{j},\ldots,k-p)$ and $\tau'(1)=1$.
 The signature of $\rho$ is $(-1)^{j-1}$. Hence,
\begin{equation}\label{varphidef}\varphi_p(e)=
 \sum_{j=1}^{k-p}
 (-1)^{j-1} \  e_{i_1}
{\cdot} \ldots {\cdot} e_{i_p} {\cdot} e_{j} \otimes e_{1} \wedge \ldots
  \widehat{e_{j}}  \ldots \wedge e_{k-p}.
\end{equation}
Proceeding  like in the computation of $\varphi_{p}(e)$, we obtain that
\begin{equation}\label{psidef}
\psi_p\varphi_p(e)= \frac{k-p}{k}\ e +
\sum_{j=1}^{k-p} \sum_{t=1}^{p}\frac{(-1)^{j-1}}{k} \quad \begin{array}{l} e_{i_1} {\cdot} \ldots
 \widehat{e_{i_t}} \ldots {\cdot} e_{ i_p}\cdot
e_{j} \otimes \\ \otimes e_{i_t}\wedge e_{1} \wedge \ldots
\widehat{e_{j}}  \ldots\wedge e_{{k-p}}. \end{array}
\end{equation}
Therefore,
$$\varphi_p\psi_p\varphi_p(e)= \frac{k-p}{k}\ \varphi_p(e) + \frac{1}{k}\varphi_p(y)$$
where
$$ y=\sum_{j=1}^{k-p} \sum_{t=1}^{p} (-1)^{j-1}  e_{i_1} {\cdot} \ldots
 \widehat{e_{i_t}} \ldots {\cdot} e_{ i_p}\cdot
e_{j} \otimes e_{i_t}\wedge e_{1} \wedge \ldots
\widehat{e_{j}}  \ldots\wedge e_{{k-p}}.
$$
Using \eqref{varphidef}, we have
\begin{eqnarray*}
  \varphi_p(y) &= & \sum_{j=1}^{k-p} \sum_{t=1}^{p} (-1)^{j-1}  e_{i_1} {\cdot} \ldots
 {\cdot} e_{ i_p} \cdot
e_{j} \otimes e_{1} \wedge \ldots
\widehat{e_{j}}  \ldots\wedge e_{{k-p}}
  \\ &+& \sum_{j=1}^{k-p} \sum_{t=1}^{p}\sum_{l=1}^{j-1} (-1)^{j-1+l}  \quad \begin{array}{l} e_{i_1} {\cdot} \ldots
 \widehat{e_{i_t}} \ldots {\cdot} e_{ i_p}\cdot
e_{j}\cdot e_l \\ \otimes e_{i_t}\wedge e_{1} \wedge \ldots \widehat{e_l} \ldots
\widehat{e_{j}}  \ldots\wedge e_{{k-p}}\end{array} \\
&+& \sum_{j=1}^{k-p} \sum_{t=1}^{p}\sum_{l=j+1}^{k-p} (-1)^{j+l} \quad \begin{array}{l}  e_{i_1} {\cdot} \ldots
 \widehat{e_{i_t}} \ldots {\cdot} e_{ i_p}\cdot
e_{j}\cdot e_l \\ \otimes e_{i_t}\wedge e_{1} \wedge \ldots \widehat{e_j} \ldots
\widehat{e_{l}}  \ldots\wedge e_{{k-p}} \end{array} \\ &=& p \ \varphi_p(e)
\end{eqnarray*}
since the two big summands on the indices $l$ and $j$ cancel.
Hence, we have seen that
$$\varphi_p\psi_p\varphi_p(e)= \frac{k-p}{k}\ \varphi_p(e) + \frac{p}{k}\varphi_p(e)=\varphi_p(e).$$
\end{proof}

Although the Koszul complex is algebraic split, the short exact sequences $\mu^j(\Psi^k(\overline{E})^*)$ are not
orthogonal split and hence, its Bott-Chern form is not zero, as it would be desirable. In order to achieve this, we need to multiply each morphism
 $\varphi_p$ of the Koszul complex by $\frac{1}{\sqrt{k}}$.

\begin{prp}\label{koszulvanish}
Let $\overline{E}$ be a hermitian vector bundle over a smooth proper complex variety and consider the Koszul complex $\Psi^k(\overline{E})^*$, $k\geq 1$. Then, the acyclic cochain complex $\lambda_k(\Psi^k(E)^*)$ satisfies that $\mu^j(\lambda_k(\Psi^k(\overline{E})^*))$ is hermitian split for all $j$.
\end{prp}
\begin{proof}
Let $\iota_{p}$ and $j_{p}$ be the inclusions defined in \eqref{iota}, $\varphi_{p}$ as defined in \eqref{varphi} and $\psi_{p}$ be defined as in the proof of the previous lemma.

Let us compute explicitly the squared norm of an element of  $\im \varphi_p$ by the two hermitian metrics: the one induced by the inclusion into $S^{p+1}\overline{E} \otimes \bigwedge\nolimits^{k-p-1} \overline{E}$ and the quotient metric induced by that of $S^p\overline{E} \otimes \bigwedge\nolimits^{k-p} \overline{E}$.
To prove the statement, we need to see that the two norms are related by the factor $1/k$.

Denote by $||\varphi_p(e)||_{i}$ the norm of $\varphi_p(e)$ in $S^{p+1}\overline{E} \otimes \bigwedge\nolimits^{k-p-1} \overline{E}$ and by $||\varphi_p(e)||_{q}$ the norm of $\varphi_p(e)$ given by considering it as a quotient of  $S^{p}\overline{E} \otimes \bigwedge\nolimits^{k-p}\overline{E}$.

Assume that the rank of $\overline{E}$ is $n$ and consider a orthonormal local frame in $\overline{E}$, $e_{1},\dots,e_{n}$. Then, $\{e_{i_{1}}\otimes \ldots \otimes e_{i_{k}}\}_{i_{1},\dots,i_{k}\in \{1,\dots,n\}}$ forms a orthonormal basis of $T^k(\overline{E})$. To ease the notation, let us write
$$e^{\otimes}_{i_{1},\dots,i_{k}}=e_{i_{1}}\otimes \ldots \otimes e_{i_{k}}.$$
For an element of the form
$\sum_{\lambda \in I} \alpha_{\lambda} e^{\otimes}_{\lambda} $
with $I$ a subset of $\{1,\dots,n\}^{k}$ and $ \alpha_{\lambda}\in \C$, the square of its norm is given by
$\sum_{\lambda \in I} \alpha_{\lambda}^2. $

\emph{Notation:} If $\sigma\in \mathfrak{G}_p$ and $\lambda=(\lambda_1,\dots,\lambda_p)\in \{1,\dots,n\}^{p}$ we will write
$$\sigma(\lambda_1,\dots,\lambda_p) = (\lambda_{\sigma(1)},\dots,\lambda_{\sigma(p)})$$

Renaming the indices, it is only necessary to compute the norms for an element of the form
$$\varphi_{p}(e) \quad\textrm{with}\quad e= e_{i_{1}}\cdot \ldots \cdot e_{i_{p}} \otimes e_{1} \wedge \ldots \wedge e_{k-p}$$
where $i_{1},\dots,i_{p}\in \{1,\dots,n\}$. Let $m_{i}$ denote the number of times that $e_{i}$ appears in $e_{i_{1}}\cdot \ldots \cdot e_{i_{p}} $. In this way, we have an equality
$$ e= e_{1}^{\cdot m_{1}}\cdot \ldots \cdot e_{n}^{\cdot m_{n}} \otimes e_{1} \wedge \ldots \wedge e_{k-p},$$
where $ e_{i}^{\cdot m_{i}}= e_{i}\cdot \stackrel{m_{i}}{\ldots} \cdot e_{i}$. Note that $m_{1}+\dots + m_{n}=p$.

Then, by \eqref{varphidef}, we have
$$\varphi_p(e)=
 \sum_{j=1}^{k-p}
 (-1)^{j-1} \  e_{i_1}
{\cdot} \ldots {\cdot} e_{i_p} {\cdot} e_{j} \otimes e_{1} \wedge \ldots
  \widehat{e_{j}}  \ldots \wedge e_{k-p}.
$$
The norm $||\varphi_p(e)||_i$ is computed as the norm of $\frac{(\iota_{p+1}\otimes j_{k-p-1})}{\sqrt{(p+1)!(k-p-1)!}}(\varphi_p(e))$ in $T^k(\overline{E})$. The  term $(\iota_{p+1}\otimes j_{k-p-1})(\varphi_p(e))$ is
$$ \sum_{\substack{\sigma\in \mathfrak{S}_{p+1} \\ \tau\in \mathfrak{S}_{k-p-1}}}(-1)^{|\tau|}
 \sum_{j=1}^{k-p}  (-1)^{j-1} e^{\otimes}_{\sigma(i_1,\ldots,i_p,j),\tau(1,\ldots,\widehat{j},\ldots,k-p)}.
$$
Its squared norm is given by the sum of the square of the number of occurrences of each different summand.
Given two different permutations $\tau,\tau' \in \mathfrak{S}_{k-p-1}$, we have
$$e^{\otimes}_{\sigma(i_1,\ldots,i_p,j),\tau(1,\ldots,\widehat{j},\ldots,k-p)}\neq e^{\otimes}_{\sigma(i_1,\ldots,i_p,j),\tau'(1,\ldots,\widehat{j},\ldots,k-p)}.$$
Moreover, if $j\neq j'$, the terms obtained are different as well. The only repetitions will come from the permutations $\sigma,\sigma' \in   \mathfrak{S}_{p+1}$
satisfying $\sigma(i_1,\ldots,i_p,j)=\sigma'(i_1,\ldots,i_p,j)$. For a fixed $\tau$ and $j$, there are $\frac{(p+1)!}{ m_{1}!\dots (m_{j}+1)!\dots m_{n}!}$ different terms, each of them appearing $m_{1}!\dots (m_{j}+1)!\dots m_{n}!$ times  (and   with the same sign).
Therefore,
\begin{eqnarray*}
||\varphi_p(e)||_{i}^{2} &=&  \frac{ (k-p-1)!}{(p+1)!(k-p-1)!} \sum_{j=1}^{k-p}  \frac{(p+1)! (m_{1}!\dots (m_{j}+1)!\dots m_{n}!)^{2} }{ m_{1}!\dots (m_{j}+1)!\dots m_{n}!} \\
&=&  \sum_{j=1}^{k-p} m_{1}!\dots (m_{j}+1)!\dots m_{n}! =
  m_{1}!\dots m_{n}! \
\sum_{j=1}^{k-p} (m_{j}+1) \\ &=&    m_{1}!\dots m_{n}! \
( k-p + \sum_{j=1}^{k-p} m_{j}).
\end{eqnarray*}

Let us proceed now to the computation of $||\varphi_p(e)||_{q}$. This norm is given by
$||\frac{1}{\sqrt{p!(k-p)!}}(\iota_{p}\otimes j_{k-p})(w)||_{T^k(\overline{E})}$, where $w\in (\ker \varphi_{p})^{\perp}$ satisfies $\varphi_p(w)=\varphi_p(e)$.
Let us see that $w=\psi_{p}\varphi_p(e)$. Since we have already seen that $\varphi_{p}\psi_{p}\varphi_{p}(e)=\varphi_{p}(e)$, it is enough to check that $\psi_{p}\varphi_{p}(e)\in  (\ker \varphi_{p})^{\perp}$.

By \eqref{psidef}, we have
\begin{eqnarray*}
k\psi_p\varphi_p(e)&=& (k-p)e +
\sum_{j=1}^{k-p} (-1)^{j-1} \sum_{t=1}^{p} \quad \begin{array}{l} e_{i_1} {\cdot} \ldots
 \widehat{e_{i_t}} \ldots {\cdot} e_{ i_p}\cdot
e_{j} \otimes \\ \phantom{hola}\otimes e_{i_t}\wedge e_{1} \wedge \ldots
\widehat{e_{j}}  \ldots\wedge e_{{k-p}} \end{array}
\end{eqnarray*}
If $i_{t}\in \{1,\dots,k-p\}$ and $i_{t}\neq j$, we have $ e_{i_t}\wedge e_{1} \wedge \ldots
\widehat{e_{j}}  \ldots\wedge e_{{k-p}}=0$. If $i_{t}=j$, then
$$e_{i_1} {\cdot} \ldots
 \widehat{e_{i_t}} \ldots {\cdot} e_{ i_p}\cdot
e_{j}\otimes e_{i_t}\wedge e_{1} \wedge \ldots
\widehat{e_{j}}  \ldots\wedge e_{{k-p}}=(-1)^{j-1} e.$$
Hence,
{\small
\begin{eqnarray*}
k\psi_p\varphi_p(e)&=& \big(k-p+\sum_{j=1}^{k-p}m_{j}\big)e + \\ &+&
\sum_{j=1}^{k-p} \sum_{\substack{t\in \{k-p+1,\dots,n\} \\ m_{t\neq 0}}} (-1)^{j-1}  m_{t} \left(\begin{array}{l}
e_{1}^{\cdot m_{1}} \dots e_{j}^{\cdot m_{j}+1} \dots e_{t}^{\cdot m_{t}-1} \dots e_{n}^{\cdot m_{n}}\otimes \\\phantom{holahola}\otimes  e_{t}\wedge e_{1} \wedge \ldots
\widehat{e_{j}}  \ldots\wedge e_{{k-p}}\end{array}\right).
\end{eqnarray*}}

Let $\Phi=(\iota_{p}\otimes j_{k-p})$.
Denote by $A= \frac{1}{k}\big(k-p+\sum_{j=1}^{k-p}m_{j}\big)\Phi(e)$ and $B=\Phi(\psi_p\varphi_p(e)) - A$. We have:
\begin{eqnarray*}
A &=&  \frac{1}{k}\big(k-p+\sum_{j=1}^{k-p}m_{j}\big)  \sum_{\substack{\sigma\in \mathfrak{S}_p \\ \tau\in \mathfrak{S}_{k-p}}}(-1)^{|\tau|} e^{\otimes}_{\sigma(i_1,\ldots,i_p),\tau(1,\ldots,k-p)}.
\\
B &=&  \frac{1}{k}   \sum_{j=1}^{k-p} \sum_{\substack{t\in \{k-p+1,\dots,n\}\\ m_{t\neq 0}}}(-1)^{j-1}  \sum_{\substack{\sigma\in \mathfrak{S}_p \\ \tau\in \mathfrak{S}_{k-p}}}(-1)^{|\tau|}
 \\ && \qquad \qquad \qquad  m_{t}\
e^{\otimes}_{\sigma(1^{m_1},\dots,j^{m_j+1},\dots,t^{m_t-1},\dots,n^{m_n})}\otimes e^{\otimes}_{\tau(t,1,\dots,\widehat{j},\ldots,k-p)}.
\end{eqnarray*}
We want to see that $(A+B)$ belongs to $\Phi(\ker \varphi_{p})^{\perp}=\Phi(\im \varphi_{p-1})^{\perp}$ in $T^k(\overline{E})$.
Let  $v=\varphi_{p-1}(f)$, where $f=e_{r_1}{\cdot} \ldots {\cdot} e_{r_{p-1}}  \otimes e_{s_1} \wedge \ldots
  \wedge e_{s_{k-p+1}}$, with $s_1<\dots<s_{k-p+1}$. Then,
$$v= \sum_{t=1}^{k-p+1} (-1)^{t-1} \  e_{r_1} {\cdot} \ldots {\cdot} e_{r_{p-1}} {\cdot} e_{s_t} \otimes e_{s_1} \wedge \ldots
  \widehat{e_{s_t}}  \ldots \wedge e_{s_{k-p+1}},$$
and so
$$\Phi(v)=\sum_{t=1}^{k-p+1} (-1)^{t-1} \sum_{\substack{\sigma\in \mathfrak{S}_p \\ \tau\in \mathfrak{S}_{k-p}}}(-1)^{|\tau|} e^{\otimes}_{\sigma(r_1,\ldots,r_{p-1},s_t),\tau(s_1,\ldots,\widehat{s_t},\ldots,s_{k-p+1})}. $$

 It is straightforward to see that the scalar product
  $\langle\Phi(v),A\rangle$ is 0, unless
$f=e_{r_1}{\cdot} \ldots {\cdot} e_{r_{p-1}}  \otimes e_{1} \wedge \ldots
  \wedge e_{k-p} \wedge e_{s}$  with $\{i_1,\dots,i_p\} = \{r_1,\dots,r_{p-1},s\}$ and $s\in \{k-p+1,\dots,n\}$.
  For $B$, we have two situations where it is not obvious that $\langle\Phi(v),B\rangle$ is 0. The first of them is the same as above. The second is the case where there exists an index $j$ such that $\{1,\ldots,\widehat{j},\ldots,k-p\}=\{s_1,\ldots,s_{k-p-1}\}$, $s_{k-p},s_{k-p+1}>k-p$ and
  $\{r_1,\dots,r_{p-1}\}=\{i_1,\dots,i_p,j\} \setminus \{s_{k-p},s_{k-p+1}\}$.

  In the first scenario (  $f=e_{r_1}{\cdot} \ldots {\cdot} e_{r_{p-1}}  \otimes e_{1} \wedge \ldots
  \wedge e_{k-p} \wedge e_{s}$  with $\{i_1,\dots,i_p\} = \{r_1,\dots,r_{p-1},s\}$ and $s\in \{k-p+1,\dots,n\}$), we have
  \begin{eqnarray*}
    \langle\Phi(v),A\rangle&=& \alpha\cdot \sum_{\substack{\sigma,\sigma'\in \mathfrak{S}_p\\ \tau,\tau'\in \mathfrak{S}_{k-p}}}(-1)^{|\tau|+|\tau'|}
    \left(\begin{array}{l} \langle e^{\otimes}_{\sigma(i_1,\ldots,i_p),\tau(1, \ldots,k-p)}, \\ \quad  e^{\otimes}_{\sigma'(r_1,\ldots,r_{p-1},s),\tau'(1, \ldots,k-p)} \rangle\end{array} \right)\\
     &=& \alpha\cdot \sum_{\tau\in  \mathfrak{S}_{k-p}} \sum_{\sigma,\sigma'\in \mathfrak{S}_p}
    \begin{array}{l}  \langle  e^{\otimes}_{\sigma(i_1,\ldots,i_p),\tau(1, \ldots,k-p)}, \\ \quad  e^{\otimes}_{\sigma'(r_1,\ldots,r_{p-1},s),\tau(1, \ldots,k-p)}\rangle \end{array}
      \end{eqnarray*}
  with $\alpha= \frac{(-1)^{k-p}}{k}\big(k-p+\sum_{j=1}^{k-p}m_{j}\big) $.
 Note that for every $\sigma\in \mathfrak{S}_p$, the number of $\sigma'\in \mathfrak{S}_p$ such that $$\langle e^{\otimes}_{\sigma(i_1,\ldots,i_p),\tau(1, \ldots,k-p)}, e^{\otimes}_{\sigma'(r_1,\ldots,r_{p-1},s),\tau(1, \ldots,k-p)} \rangle=1$$ is the same. This number is $\lambda=m_1!\dots m_n!$. Then,
\begin{eqnarray*}
     \langle\Phi(v),A\rangle &=&  \alpha \cdot \sum_{\tau\in  \mathfrak{S}_{k-p}} \sum_{\sigma\in \mathfrak{S}_p}
    \lambda =  \frac{(-1)^{k-p}}{k} \big(k-p+\sum_{j=1}^{k-p}m_{j}\big)  (k-p)! p! \lambda \\
    &=&   \frac{(-1)^{k-p}}{k}(k-p)! p!\big(k-p+\sum_{j=1}^{k-p}m_{j}\big)\lambda.
      \end{eqnarray*}
We proceed in the same way for $B$:
\begin{eqnarray*}
     \langle\Phi(v),B\rangle &=&
    \frac{1}{k}  \sum_{j=1}^{k-p} \sum_{\substack{t\in \{k-p+1,\dots,n\}\\ m_{t\neq 0}}}  \sum_{\substack{\sigma,\sigma'\in \mathfrak{S}_p \\ \tau,\tau'\in \mathfrak{S}_{k-p}}} \sum_{l=1}^{k-p+1} (-1)^{j+l} (-1)^{|\tau|+|\tau'|}m_{t}\
 \\ &&
\quad \langle e^{\otimes}_{\sigma(1^{m_1},\dots,j^{m_j+1},\dots,t^{m_t-1},\dots,n^{m_n}),\tau(t,1,\dots,\widehat{j},\ldots,k-p)},  \\ && \hspace{4cm} e^{\otimes}_{\sigma'(r_1,\ldots,r_{p-1},l),\tau'(1, \ldots,\widehat{l},\ldots,s)} \rangle
\\ &=&   \frac{1}{k} \sum_{j=1}^{k-p} \sum_{\substack{\sigma,\sigma'\in \mathfrak{S}_p \\ \tau\in \mathfrak{S}_{k-p}}}(-1)^{k-p+1} m_{s}\
\\ &&
\quad \langle e^{\otimes}_{\sigma(1^{m_1},\dots,j^{m_j+1},\dots,s^{m_s-1},\dots,n^{m_n}),\tau(1,\dots,\widehat{j},\ldots,k-p,s)},   \\ && \hspace{4cm}  e^{\otimes}_{\sigma'(r_1,\ldots,r_{p-1},j),\tau(1, \ldots,\widehat{j},\ldots,s)}\rangle
\end{eqnarray*}
\begin{eqnarray*}
\phantom{\langle\Phi(v),B\rangle}&=& \frac{1}{k} (-1)^{k-p+1}  m_{s} (k-p)! p!  \sum_{j=1}^{k-p} m_{1}!\dots (m_{j}+1)! \dots (m_{s}-1)! \dots m_{n}! \\
&=&
\frac{ (k-p)! p!}{k } (-1)^{k-p+1}  \lambda  (k-p+\sum_{j=1}^{k-p} m_{j})
\\ &=& -   \Phi(v)\cdot A.  \end{eqnarray*}
Therefore,
$$\langle \Phi(v),A+B \rangle=0$$
in $T^k(\overline{E})$ for all $v\in \im \varphi_{p-1}$ of the first form.

Let us consider now the second non-trivial scenario, that is, assume there exists an index $j$ such that $\{1,\ldots,\widehat{j},\ldots,k-p\}=\{s_1,\ldots,s_{k-p-1}\}$, we have $s_{k-p},s_{k-p+1}>k-p$, and
  $\{r_1,\dots,r_{p-1}\}=\{i_1,\dots,i_p,j\} \setminus \{s_{k-p},s_{k-p+1}\}$. Then, $  \langle\Phi(v),A\rangle=0$  and we have (with $\beta$ a constant):
\begin{eqnarray*}
     \langle\Phi(v),B\rangle &=&
    \beta\   \sum_{\substack{t\in \{k-p+1,\dots,n\}\\ m_{t\neq 0}}}\sum_{l=1}^{k-p+1}  \sum_{\substack{\sigma,\sigma'\in \mathfrak{S}_p \\ \tau,\in \mathfrak{S}_{k-p}}} m_{t}\
 \\ &&
 \quad \langle e^{\otimes}_{\sigma(1^{m_1},\dots,j^{m_j+1},\dots,t^{m_t-1},\dots,n^{m_n}),\tau(1,\dots,\widehat{j},\ldots,k-p,t)},  \\ && \hspace{4cm} e^{\otimes}_{\sigma'(r_1,\ldots,r_{p-1},s_l),\tau(s_1, \ldots,\widehat{s_l},\ldots,s_{k-p+1})} \rangle
\\ &=&  \beta p! (k-p)!  m_{1}!\dots (m_{j}+1)! \dots m_{n}! \big( (-1)^{k-p-1} + (-1)^{k-p}  \big)
 \\ &=& 0,
\end{eqnarray*}
where we have used that the only indices such that the scalar product is non-zero are the pairs $l=k-p$, $t=s_{k-p+1}$ and  $l=k-p+1$, $t=s_{k-p}$.

We have obtained that $\langle \Phi(v),A+B \rangle=0$
in $T^k(\overline{E})$ for all $v\in \im \varphi_{p-1}$ as desired.

We proceed now to compute the norm  $||\varphi_p(e)||_q$ which is given by $$\left\|\frac{1}{\sqrt{p!(k-p)!}}(A+B)\right\|_{T^k(\overline{E})}.$$
 As above, we should group the different summands of $A$ and $B$ if they are the same.
In order to do this, observe that the terms obtained for different permutations $\tau$, different $j$ or different $t$ are not equal. Moreover, the summands in $A$ are all different from the summands in $B$. Therefore, the only repetitions are obtained by the permutations $\sigma \in \mathfrak{S}_{p}$.

With these observations we obtain, as in the computation of $||\cdot ||_{i}$:
\begin{eqnarray*}
\frac{||A||_{q}^2}{p!(k-p)!} &=&   \frac{1}{k^2 }\big(k-p+\sum_{j=1}^{k-p}m_{j}\big)^2 m_{1}!\dots m_{n}!
\\
\frac{||B||_{q}^2}{p!(k-p)!} &=&   \frac{1}{k^2p!(k-p)!} \sum_{j=1}^{k-p} \sum_{\substack{t\in \{k-p+1,\dots,n\}\\ m_{t\neq 0}}}\sum_{\tau\in \mathfrak{S}_{k-p}} \\ &&\hspace{1.6cm}\frac{p! (m_{t})^2 (m_{1}!\dots (m_{j}+1)! \dots (m_{t}-1)! \dots m_{n}! )^2}{m_{1}!\dots (m_{j}+1)! \dots (m_{t}-1)! \dots m_{n}! } \\ &=&
\frac{1}{k^2} \sum_{j=1}^{k-p} \sum_{\substack{t\in \{k-p+1,\dots,n\}\\ m_{t\neq 0}}}  m_{t} \ m_{1}!\dots (m_{j}+1)! \dots m_{n}!
\end{eqnarray*}
Therefore,
\begin{eqnarray*}
||\varphi_p(e)||_{q}^2 &=& \frac{ m_{1}!\dots m_{n}! }{k^2  } \Big[
\big(k-p+\sum_{j=1}^{k-p}m_{j}\big)^2+ \sum_{j=1}^{k-p} \sum_{t=k-p+1}^{n} m_{t}  (m_{j}+1)
\Big] \\ &=&  \frac{  m_{1}!\dots m_{n}! }{k^2 } \Big[
\big(k-p+\sum_{j=1}^{k-p}m_{j}\big)^2+ \big(\sum_{j=1}^{k-p}  (m_{j}+1) \big) \big(\sum_{t=k-p+1}^{n} m_{t}\big) \Big].
\end{eqnarray*}
Note that
$$\sum_{t=k-p+1}^{n} m_{t}=p-\sum_{j=1}^{k-p} m_{j}. $$
Hence, denoting $\beta= \frac{ m_{1}!\dots m_{n}! }{k^2 }$, we have
\begin{eqnarray*}
||\varphi_p(e)||_{q}^2 &=& \beta \Big[
\big(k-p+\sum_{j=1}^{k-p}m_{j}\big)^2+ \big(k-p+\sum_{j=1}^{k-p} m_{j} \big) \big(p-\sum_{j=1}^{k-p} m_{j}\big) \Big] \\
&=& \beta k
\big(k-p+\sum_{j=1}^{k-p}m_{j}\big)
=  \frac{ m_{1}!\dots m_{n}! }{k  }\big(k-p+\sum_{j=1}^{k-p}m_{j}\big).
\end{eqnarray*}
Therefore, $$\frac{||\varphi_p(e)||_{i}^2}{ ||\varphi_p(e)||_{q}^2 }= k. $$
It follows that if we define $\varphi_p'=\frac{1}{\sqrt{k}}\varphi_p$ we
have
\begin{eqnarray*}
||\varphi_p'(e)||_{i}^2 &=&  \frac{1}{k} ||\varphi_p(e)||_{i}^2 =  ||\varphi_p(e)||_{q}^2 = ||\varphi'_p(e)||_{q}^2 .
\end{eqnarray*}
The last equality follows from the fact that if  $w\in (\ker \varphi_{p})^{\perp}$ satisfies $\varphi_p(w)=\varphi_p(e)$ then
$w\in (\ker \varphi'_{p})^{\perp}$ and $\varphi'_p(w)=\varphi'_p(e)$.
\end{proof}

\begin{rem}
  Note that if we had defined from the very beginning the arrows $\varphi_p$ of the Koszul complex $\Psi^K$ to be $\frac{1}{\sqrt{k}}\varphi_p$, then
  we would not have the isometry of chain complexes \eqref{koszulisosum}.
\end{rem}

\begin{crl}\label{chkoszul} Let  $\overline{E}$ be a hermitian vector bundle. Then, for all $0\leq j\leq k-1$ we have
  $$d_{\mmD}\ch(\mu^j(\Psi^k(\overline{E})^*))=0.$$
   \end{crl}
  \begin{proof}
    Consider the commutative diagram of short exact sequences
   {\small $$ \xymatrix{ 0 \ar[r] & \ker \varphi_p \ar[r] \ar[d]_{=} & S^{p}\overline{E} \otimes \bigwedge^{k-p} \overline{E} \ar[r]^{\varphi_p} \ar[d]_{=} &   \ker \varphi_{p+1} \ar[r] \ar[d]_{\frac{1}{\sqrt{k}}} & 0 \\ 0 \ar[r] & \ker (\frac{1}{\sqrt{k}}\varphi_p) \ar[r] & S^{p}\overline{E} \otimes \bigwedge^{k-p} \overline{E} \ar[r]^{\frac{1}{\sqrt{k}}\varphi_p} &
        \ker(\frac{1}{\sqrt{k}}\varphi_{p+1}) \ar[r] & 0.  }  $$  }
By last proposition, $\ch(\mu^j(\lambda_k\Psi^k(\overline{E})^*))=0$. Hence
    \begin{eqnarray*}
d_{\mmD}\ch(\mu^j(\Psi^k(\overline{E})^*))  &=& d_{\mmD}\ch(\mu^j(\lambda_k\Psi^k(\overline{E})^*)) - d_{\mmD}\ch(id_{\ker \varphi_p}) \\ && +
d_{\mmD}\ch(id_{S^{p}\overline{E} \otimes \bigwedge^{k-p}} ) - d_{\mmD}\ch(\ker \varphi_{p+1} \xrightarrow{\frac{1}{\sqrt{k}}}  \ker \varphi_{p+1}   ) \\
&=& 0.
\end{eqnarray*}
  \end{proof}

\subparagraph{Adams operations and the secondary Euler characteristic.}
Let $E$  be a locally free sheaf of finite rank and let
$$\psi^k(E)=N_k(\lambda^1(E),\ldots,\lambda^k(E)),$$
with $N_k$ being the $k$-th Newton polynomial. These are the Adams
operations associated to the lambda operations $\lambda^k$ on
 locally free sheaves of finite rank.
Let $\Psi^k(E)$ be the secondary Euler characteristic
class of the Koszul complex of $E$. As shown by Grayson in \cite[$\S$3]{Grayson1}, the secondary Euler characteristic class
of the $k$-th Koszul complex agrees, in the quotient group $K_0(X)$, with the usual $k$-th Adams operation. Therefore, in $K_0(X)$,
$$\psi^k(E)=\Psi^k(E).$$
This means that there exist short exact sequences $s_1,\dots,s_r$
such that
$$\psi^k(E)-\Psi^k(E)=\sum_{i=1}^r d(s_i).$$
In the next proposition, we construct explicitly  such a set of short exact sequences.
For that, let $\Psi^{k}(E)^{t*}$ denote the Koszul complex obtained by changing the $p$-th component $S^pE\otimes \bigwedge\nolimits^{k-p}E$ by
$\bigwedge\nolimits^{k-p}E\otimes S^pE$ via the canonical isomorphism.

\begin{prp}\label{psicomp}
  Let $E$ be a  locally free sheaf of finite rank. Then,
  $$\psi^k(E)-\Psi^k(E)=\sum_{i=1}^r d(s_i)$$
  with $s_i$ being $\mu^p(\Psi^{k_i}(E)^{t*})\otimes A_i,$ with $A_i$ some  locally free sheaves of finite rank of the form $\bigwedge\nolimits^{j_i}{E}$ or $T^{j_i}E$ and some indices $k_i,p,j_i$ or $s_i$ being the canonical isomorphisms $\bigwedge\nolimits^{k-p}E\otimes S^pE \cong S^pE\otimes \bigwedge\nolimits^{k-p}E$.
\end{prp}
\begin{proof}
  Consider the polynomial relating the lambda and Adams operations:
\begin{equation}\label{psi}\psi^k= \psi^{k-1}\lambda^1 -\psi^{k-2}\lambda^2 + \dots + (-1)^{k-1} k \lambda^k.
\end{equation}

 For two linear combinations of  locally free sheaves of finite rank, write $A \simeq B$ if there exist  short exact sequences $s_1,\dots,s_l$, in the form of the statement of the proposition, such that
$$A -B =\sum_{i=1}^l d(s_i).$$

Let $L_k=\{(i_1,\dots,i_l)|\ i_j\in \{1,\dots,k\}, i_1+\dots+i_l=k\}$ be the set of partitions of $k$ and let $S^k$ denote the symmetric product. We will show, first of all, that
$$S^k\simeq \sum_{(i_1,\dots,i_l)\in L_k } (-1)^{l+k} \lambda^{i_1}\dots \lambda^{i_l}.$$
We prove this result by induction on $k$. For $k=1$, the statement is obvious since $S^1=\lambda^1=id$.
Assume that the result is true up to $k$.
Considering the $(k+1)$-th Koszul complex $\Psi^{k+1}(E)^{t*}$, we have
$$S^{k+1} \simeq  \sum_{i=1}^{k+1} (-1)^{i+1} \lambda^i\otimes S^{k+1-i}.$$
By induction hypothesis,
\begin{eqnarray*}
S^{k+1} & \simeq  & \ \sum_{i=1}^{k+1} (-1)^{i+1} \lambda^i\otimes \left(\sum_{(i_1,\dots,i_l)\in L_{k+1-i} } (-1)^{l+k+1-i} \lambda^{i_1}\dots \lambda^{i_l}  \right) \\
& \simeq & \sum_{(i,i_1,\dots,i_l)\in L_{k+1} } (-1)^{(l+1)+(k+1)}  \lambda^i  \lambda^{i_1}\dots \lambda^{i_l}
\end{eqnarray*}
as desired.

Next, observe that by definition $\Psi^1=\psi^1$. Hence, if we show that $\Psi^k$ satisfies the recursive formula \eqref{psi}, up to short exact sequences of the desired form, we are done.
By definition,
\begin{eqnarray*}
\Psi^k(E) &=&  \sum_{p\geq 0} (-1)^{k-p+1} (k-p) S^p(E) \otimes \bigwedge\nolimits^{k-p}E \\ &\simeq & \sum_{p\geq 0} (-1)^{k-p+1} (k-p)\lambda^{k-p}E\otimes S^pE.
\end{eqnarray*}
Hence, using the previous relation for $S^k$, we have (omitting the writing of the locally free sheaf $E$):
\begin{eqnarray*}
  \Psi^k &\simeq & \sum_{p\geq 0} (-1)^{k-p+1} (k-p)\lambda^{k-p}\otimes  \left(
  \sum_{(i_1,\dots,i_l)\in L_p } (-1)^{l+p} \lambda^{i_1}\dots \lambda^{i_l}
  \right) \\
  &\simeq & \sum_{p\geq 0} \sum_{(i_1,\dots,i_l)\in L_p }(-1)^{k+l+1} (k-p)\lambda^{k-p} \lambda^{i_1}\dots \lambda^{i_l}
  \\ &\simeq &\sum_{(i_1,\dots,i_{l+1})\in L_{k}}(-1)^{k+l+1} i_1\lambda^{i_1} \lambda^{i_2}\dots \lambda^{i_{l+1}}
  \\ & \simeq & (-1)^{k+1} k \lambda^k  + \sum_{s=1}^{k-1} (-1)^s \left(\sum_{(i_1,\dots,i_{l})\in L_{k-s} }(-1)^{k+l+1-s} i_1\lambda^{i_1} \dots \lambda^{i_{l}} \right)\lambda^s \\
  & \simeq & (-1)^{k+1} k \lambda^k  - \sum_{s=1}^{k-1} (-1)^s \Psi^{k-s}\lambda^s.
  \end{eqnarray*}
  This finishes the proof of the proposition.
\end{proof}

\begin{rem}
  Note that last proposition applies to any suitable exact category where Grayson's Adams operations are defined, that is, where the correct notion of symmetric, exterior and tensor product is available (see \cite{Grayson1}).
  \end{rem}

Let $X$ be a proper arithmetic variety over $\Z$. Let
$$\Psi^k : \mmD^{2p-*}(X,p) \rightarrow \mmD^{2p-*}(X,p)$$
be the morphism that maps $\alpha$ to $k^p \alpha.$ That is, we
endow $\bigoplus_{p\geq 0}\mmD^{2p-*}(X,p)$ with the canonical
$\lambda$-ring structure  corresponding
to the graduation given by $p$.

\begin{prp}\label{gs} Let $X$ be a proper arithmetic variety and
let $\overline{E}$ be a hermitian vector bundle over $X$. Then,
$$\Psi^k\ch(\overline{E})=\ch\Psi^k(\overline{E})$$ in the group $\bigoplus_{p\geq
0}\mmD^{2p}(X,p)$.
\end{prp}
\begin{proof} In \cite[Lemma 7.3.3]{GilletSouleClassesII} Gillet and Soul{\'e}
proved that
$\lambda^k\ch = \ch \lambda^k$ from where it follows that
$$\psi^k(\ch(\overline{E}))=\ch(\psi^k(\overline{E})). $$
Observe that by definition,
$$\Psi^k(\ch(\overline{E}))=\psi^k(\ch(\overline{E})).$$

By proposition \ref{psicomp}, there are short exact sequences $\overline{s}_i$ being $\mu^p(\Psi^{k_i}(\overline{E})^{t*})\otimes \overline{A}_i,$ with $\overline{A}_i$ some locally free sheaves of the form $\bigwedge\nolimits^{j_i}{\overline{E}}$ or $T^{j_i}\overline{E}$ and some indices $k_i,p,j_i$, or $\overline{s}_i$ being the canonical isomorphisms $\bigwedge\nolimits^{k-p}\overline{E}\otimes S^p\overline{E} \cong S^p\overline{E}\otimes \bigwedge\nolimits^{k-p}\overline{E}$, such that
$$\psi^k(\overline{E})-\Psi^k(\overline{E})=\sum_{i=1}^r d(\overline{s}_i).$$
These short exact sequences $\overline{s}_i$ are endowed with the hermitian metric induced by the hermitian metrics
$h_{\bigwedge\nolimits^*(\overline{E})},h_{T^*(\overline{E})}$ and $h_{\Psi^*_*(\overline{E})^t}$.
By Corollary \ref{chkoszul} it follows that
$$\ch(\psi^k(\overline{E}))-\ch(\Psi^k(\overline{E}))=\sum_{i=1}^r d_{\mmD}\ch(\overline{s}_i)=0$$
and hence
$$\ch(\Psi^k(\overline{E}))= \ch(\psi^k(\overline{E})) $$
and the proposition is proved.
\end{proof}

\paragraph{The Bott-Chern form of the Koszul complex.}
For any acyclic cochain complex of hermitian vector bundles,
$\overline{A}: 0 \rightarrow \overline{A}^0 \xrightarrow{f^0} \cdots \xrightarrow{f^{j-1}} \overline{A}^j  \xrightarrow{f^j}
\cdots \xrightarrow{f^{k-1}} \overline{A}^k \rightarrow 0,$
the Bott-Chern form of $\overline{A}$ is defined by
$$\ch(\overline{A})=\ch(\mu(\overline{A}))= \sum_{j\geq 0} (-1)^{j-1}\ch(\mu^j(\overline{A})).$$

\begin{prp}
  Let $X$ be a smooth complex variety and let $\overline{E}$ be a hermitian vector bundle. With the metrics on the Koszul complex
  $\Psi^k(\overline{E})^*$ induced by $h_{\bigwedge\nolimits^*(\overline{E})},h_{T^*(\overline{E})}$ and $h_{S^*(\overline{E})}$, we have
  $$\ch(\Psi^k(\overline{E})^*) = \frac{(-1)^{k+1}\log(k)}{2} \Psi^k(\ch(\overline{E}))$$
  in $\bigoplus_{p\geq 0}\mmD^{2p-1}(X,p)$.
\end{prp}

\begin{proof}
  Consider the commutative diagram of acyclic chain complexes (we omit the  hermitian vector bundle $\overline{E}$):
  {\small $$ \xymatrix@C=14pt@R=40pt{ 0 \ar[r] &  \bigwedge^{k}  \ar[r]^(0.4){\varphi_0} \ar[d]^{\sqrt{k}^k\Id} & S^{1}\otimes \bigwedge^{k-1}   \ar[r]^(0.65){\varphi_{1}} \ar[d]^{\sqrt{k}^{k-1}\Id} & \dots \ar[r] &  S^{p}  \otimes \bigwedge^{k-p}  \ar[r]^(0.65){\varphi_p}  \ar[d]^{\sqrt{k}^{k-p}\Id} &  \dots \ar[r] &  S^{k-1}  \otimes \bigwedge^{1} \ar[r]^(0.65){\varphi_{k-1}}  \ar[d]^{\sqrt{k}\Id}  & S^k \ar[r]\ar[d]^{\Id} & 0 \\ 0\ar[r] & \bigwedge^{k}  \ar[r]_(0.4){\frac{\varphi_0}{\sqrt{k}}}  & S^{1}\otimes \bigwedge^{k-1}   \ar[r]_(0.65){\frac{\varphi_1}{\sqrt{k}}} & \dots \ar[r] &  S^{p}  \otimes \bigwedge^{k-p}  \ar[r]_(0.65){\frac{\varphi_p}{\sqrt{k}}}   &  \dots \ar[r] &  S^{k-1}  \otimes \bigwedge^{1} \ar[r]_(0.65){\frac{\varphi_{k-1}}{\sqrt{k}}} & S^k \ar[r] & 0,  }  $$  }
  with the first row standing at degree $1$ and the second at degree $2$.
  Let $\ch(\sqrt{k}^{k-p}\Id\nolimits_{S^p\overline{E}\otimes \bigwedge\nolimits^{k-p}\overline{E}})$ denote
  $\ch(0\rightarrow S^p\overline{E}\otimes \bigwedge\nolimits^{k-p}\overline{E} \xrightarrow{\sqrt{k}^{k-p}\Id} S^p\overline{E}\otimes \bigwedge\nolimits^{k-p}\overline{E})$.
It follows that
\begin{eqnarray*}
  \ch(\Psi^k(\overline{E})^*) &=& \ch(\lambda_k(\Psi^k(\overline{E})^*)) +\sum_{p=0}^k (-1)^{p+1}\ch(\sqrt{k}^{k-p}\Id\nolimits_{S^p\overline{E}\otimes \bigwedge\nolimits^{k-p}\overline{E}})
  \end{eqnarray*}
  By Proposition \ref{koszulvanish}, we know that $\ch(\lambda_k(\Psi^k(\overline{E})^*))=0$. By Lemma \ref{bcconstant} below, we have that $\ch(a \Id_{\overline{F}})=-\log(a)\ch(\overline{F})$ for any hermitian vector bundle $\overline{F}$ and $a>0$ a real number. Therefore, we get
\begin{eqnarray*}
  \ch(\Psi^k(\overline{E})^*) &=& \sum_{p=0}^k (-1)^{p+1} \log(\sqrt{k}^{k-p})\ch(S^p\overline{E}\otimes \bigwedge\nolimits^{k-p}\overline{E}) \\
  &=& \sum_{p=0}^k (-1)^{p} (k-p)\log(\sqrt{k})\ch(S^p\overline{E}\otimes \bigwedge\nolimits^{k-p}\overline{E}) \\
  &=& (-1)^{k+1} \log(\sqrt{k}) \sum_{p=0}^k (-1)^{k-p+1} (k-p)\ch(S^p\overline{E}\otimes \bigwedge\nolimits^{k-p}\overline{E}) \\
  &=& (-1)^{k+1} \log(\sqrt{k}) \ch(\Psi^k(\overline{E})) = \frac{(-1)^{k+1}\log(k)}{2}\Psi^k(\ch(\overline{E})),
  \end{eqnarray*}
  where the last equality follows from Proposition \ref{gs}.

\end{proof}

\begin{lmm}\label{bcconstant} Let $X$ be a smooth complex variety, let $\overline{F}$ be a  hermitian vector bundle over $X$ and let $a$ be a non-zero complex number. Then,
  $$\ch(0\rightarrow \overline{F} \xrightarrow{ a\Id} \overline{F})= -\log(||a||)\ch(\overline{F})$$
  in $\bigoplus_{p\geq 0}\mmD^{2p-1}(X,p)$.
\end{lmm}
\begin{proof} Let $\ch(a\Id)$ denote $\ch(0\rightarrow \overline{F} \xrightarrow{ a\Id} \overline{F})$.
Consider $t=x/y$ the local coordinates on $\P^1$. By definition,
 $$\ch(a\Id)= \frac{1}{2\pi i} \int_{\P^1} \ch(\tr\nolimits_1(a\Id)) \wedge \big(\frac{1}{2} \log t\bar{t}\big).$$
 Let $h$ denote the hermitian metric of $\overline{F}$.
 Let $\xi$ be a local frame for $F$ on an open set $U$. By \cite{BurgosKuhnKramer2}, the determined local frame on $\tr\nolimits_1(a\Id)\cong p_0^*\overline{F}$
 has metric given by the matrix
  $$\frac{y\bar{y} h(\xi) + x\bar{x} ||a||^2 h(\xi) }{y\bar{y}+x\bar{x}}= \frac{y\bar{y} + x\bar{x} ||a||^2  }{y\bar{y}+x\bar{x}} h(\xi)=
\frac{ 1 + ||a||^2  t\bar{t}  }{1+t\bar{t}} h(\xi)$$
  using local projective  coordinates. Let $h_a=\frac{ 1 + ||a||^2  t\bar{t}  }{1+t\bar{t}}$.

Let $p_0,p_1: X\times \P^1 \rightarrow \P^1$ be the projections onto the first and second coordinates respectively. Since the hermitian metric on $\tr\nolimits_1(a\Id)$ is expressed as the product of a hermitian metric on the line bundle $\mathcal{O}(1)$ and the hermitian metric $h$ of $\overline{F}$, we have
$$\ch(\tr\nolimits_1(a\Id))= \ch(p_1^*\mathcal{O}(1),h_a)\wedge \ch(p_0^*\overline{F}).$$
Hence,
\begin{eqnarray*}
\ch(a\Id) & = & \frac{1}{2\pi i} \Big(\int_{\P^1} \ch(p_1^*\mathcal{O}(1),h_a) \wedge \big(\frac{1}{2} \log t\bar{t}\big) \Big)\wedge \ch(\overline{F})
\\ &=& \frac{1}{2\pi i}  \Big(\int_{\P^1} c_1(p_1^*\mathcal{O}(1),h_a) \wedge \big(\frac{1}{2} \log t\bar{t}\big)\Big)\wedge \ch(\overline{F})
\\ &=& \frac{-\log(||a||^2)}{2}  \ch(\overline{F})= -\log(||a||)\ch(\overline{F}).
\end{eqnarray*}

\end{proof}

\subsection{Adams operations and the Beilinson regulator}

We will define the Adams operations on the rational higher arithmetic
$K$-groups of $X$ from a commutative diagram of the form
\begin{equation}\label{adamsarith} \xymatrix{
N \widehat{C}_*(X) \ar[r]^(0.37){\ch} \ar[d]_{\Psi^{k}}   & \bigoplus_{p\geq 0} \mmD^{2p-*}(X,p) \ar[d]^{\Psi^k} \\
\wZ \widehat{C}_*^{\widetilde{\P}}(X) \ar[r]_(0.37){\ch} & \bigoplus_{p\geq 0}
\mmD^{2p-*}(X,p). }
\end{equation}
We proceed as follows:
\begin{enumerate}[(1)]
\item  We first define the bottom arrow $\ch:\wZ
\widehat{C}_*^{\widetilde{\P}}(X) \rightarrow \bigoplus_{p\geq 0}
\mmD^{2p-*}(X,p)$. \item We show that there are isomorphisms
\begin{eqnarray*}
H_n(s(\hch),\Q) & \cong & \widehat{K}_n(X)_{\Q}, \\
\widehat{H}_n(\wZ \widehat{C}_*^{\widetilde{\P}}(X),\ch)_{\Q} &
\cong &  \widehat{K}^{T}_n(X)_{\Q},
\end{eqnarray*}
with $\hch$ the composition $$\wZ
\widehat{C}_*^{\widetilde{\P}}(X) \xrightarrow{\ch}
\bigoplus_{p\geq 0} \mmD^{2p-*}(X,p)\rightarrow \bigoplus_{p\geq
0} \sigma_{>0}\mmD^{2p-*}(X,p) .$$
 \item We prove that the diagram
\eqref{adamsarith} is commutative.
\end{enumerate}

Let $\overline{E} \in \widehat{C}_n(X\times (\P^1)^m)$ be a
hermitian $n$-cube on $X\times (\P^1)^m$. We define
$$\ch\nolimits_{n,m}(\overline{E}):= \frac{(-1)^{n(m+1)}}
{(2\pi i)^{n+m}} \int_{(\P^1)^{n+m}}
\ch(\tr\nolimits_n(\lambda(\overline{E})))\bullet W_{n+m}\in
\bigoplus_{p\geq 0} \mmD^{2p-n-m}(X,p).$$

\begin{prp}\label{ch}
There is a chain morphism
$$\ch:  \wZ \widehat{C}^{\widetilde{\P}}_*(X) \rightarrow \bigoplus_{p\geq 0}
\mmD^{2p-*}(X,p), $$ which maps $\overline{E} \in
\widehat{C}_n(X\times (\P^1)^m)$
 to $\ch\nolimits_{n,m}(\overline{E})$. The composition
$$ K_n(X)  \rightarrow H_n(\wZ \widehat{C}^{\widetilde{\P}}_*(X),\Q)
\xrightarrow{\ch} \bigoplus_{p\geq 0} H_{\mmD}^{2p-n}(X,p),$$ is the
Beilinson regulator.

\end{prp}
\begin{proof}
First of all, observe that the map ``$\ch$'' is well defined. Indeed, if
$\overline{E}=p_i^*\overline{\mathcal{O}(1)}\otimes \overline{F}$,
then $\ch(\overline{E})\in \sigma_i\mmD^{2p-*}(X\times
(\P^1)^{n-1},p-1)+\omega_i\wedge \sigma_i\mmD^{2p-*-2}(X\times
(\P^1)^{n-1},p-1)$ and hence $\ch(\overline{E})=0$.

In order to prove that the map ``$\ch$'' is a chain morphism, observe that
``$\ch$'' factors as
$$ \wZ \widehat{C}^{\widetilde{\P}}_{n,m}(X)  \xrightarrow{\och}
\bigoplus_{p\geq 0}\wmD^{2p-n,m}_{\P}(X,p) \xrightarrow{\varphi}
\bigoplus_{p\geq 0} \mmD^{2p-n-m}(X,\R(p)),$$ where $\varphi$
is the quasi-isomorphism of Proposition \ref{qinverse} and
$\och(\overline{E})= \och_{n,m}(\overline{E})$ is defined by
$$\och_{n,m}(\overline{E})=(-1)^{n(m+1)}\ch(\tr\nolimits_n(\lambda(\overline{E})))\in
\bigoplus_{p\geq 0}\wmD^{2p}(X\times (\P^1)^{m+n},p) ,$$ for any
$\overline{E}\in \widehat{C}_n(X\times (\P^1)^m)$. Hence, it is
enough to see that $\och$ is a chain morphism.

Let $\overline{E} \in \wZ \widehat{C}_{n,m}(X)$. Since
$\ch$ is a closed differential form, we have
\begin{eqnarray*}
d_{s}(\och_{n,m}(\overline{E})) &=& (-1)^{nm} \delta
\ch(\tr\nolimits_n(\lambda(\overline{E}))) \\ &=& \sum_{i=1}^m
(-1)^{i+nm}
  \ch(\tr\nolimits_n(\lambda(\delta_{i}^1\overline{E}-\delta_{i}^0\overline{E})))\\ &&+
  \sum_{i=m+1}^{n+m}\sum_{j=0}^2(-1)^{i+j+nm} \ch(
\tr\nolimits_n(\lambda(\partial_{i-m}^j\overline{E})))
\\ &=& (-1)^n \och_{n,m-1}(\delta \overline{E}) + \och_{n-1,m}(d \overline{E}),
\end{eqnarray*}
as desired.

Finally, since by definition there is a commutative diagram
$$\xymatrix@R=6pt{  \wZ \widehat{C}_*(X) \ar[dd]_{i}^{\cong}  \ar[drr]^(.4){\ch} \\ && \bigoplus\limits_{p\geq 0}
\mmD^{2p-*}(X,p), \\
 \wZ \widehat{C}^{\widetilde{\P}}_*(X)\ar[urr]_(.4){\ch}}
 $$ the morphism $\ch$ induces the Beilinson
regulator.
\end{proof}

We have therefore constructed the bottom arrow of diagram
\eqref{adamsarith}. For the next proposition, let $\hch:\wZ
\widehat{C}_*^{\widetilde{\P}}(X) \rightarrow \bigoplus_{p\geq 0}
\sigma_{>0}\mmD^{2p-*}(X,p)$ be the composition of the morphism defined in
 Proposition \ref{ch} with the natural projection $
\bigoplus_{p\geq 0} \mmD^{2p-*}(X,p) \rightarrow \bigoplus_{p\geq
0} \sigma_{>0}\mmD^{2p-*}(X,p)$.

\begin{prp}\label{cubes2}
There are isomorphisms
\begin{eqnarray*}
\widehat{H}_n(\wZ \widehat{C}_*^{\widetilde{\P}}(X),\ch)_{\Q} &
\cong &  \widehat{K}^{T}_n(X)_{\Q},
\\ H_n(s(\hch),\Q) & \cong & \widehat{K}_n(X)_{\Q},
\end{eqnarray*}
induced by the isomorphism  $H_n(\wZ
C_*^{\widetilde{\P}}(X),\Q)\cong K_n(X)_{\Q}$  of Proposition
\ref{dold}.
\end{prp}
\begin{proof}
Both isomorphisms are a consequence of Proposition \ref{dold},
and the five lemma using the exact sequences   of Lemma
\ref{arithlong} and Proposition \ref{Delignesouleprops}.
\end{proof}

At this point, all that remains to see is that the diagram
\eqref{adamsarith}  is commutative. This will be a consequence of
the next series of lemmas and propositions.

The next lemma tells us that the morphism ``$\ch$'' maps the
split exact sequences to zero in the complex $ \bigoplus_{p\geq
0}\wmD_{\P}^{2p-*}(X,p)$.

\begin{lmm}\label{chzero} Let $X$ be a smooth proper complex variety.
Consider a split exact sequence $$\overline{E}: 0\rightarrow
\overline{E}^0\rightarrow \overline{E}^0\oplus \overline{E}^1
\rightarrow \overline{E}^1 \rightarrow 0
$$
of hermitian vector bundles over $X$. Then, in the complex $
\bigoplus_{p\geq 0}\wmD_{\P}^{2p-*}(X,p)$, it holds
$\ch(\overline{E})=0.$
\end{lmm}
\begin{proof}
Clearly, the exact sequence $\overline{E}$ already has canonical
kernels. Let us compute $\ch(\tr\nolimits_1(\overline{E}))$. By
definition, $\tr_1(\overline{E})$ is the kernel of the morphism
\begin{eqnarray*}
\overline{E}^0(1)\oplus \overline{E}^1(1) \oplus \overline{E}^1(1) & \rightarrow & \overline{E}^1(2) \\
 (a,b,c) &\mapsto & b\otimes x-c\otimes y.
\end{eqnarray*}
For every locally free sheaf $B$, there is a short exact sequence
$$ 0\rightarrow B \xrightarrow{f} B(1)\oplus B(1) \xrightarrow{g} B(2) \rightarrow 0 $$
where $f$ sends $b$ to $(b\otimes y,b\otimes x)$ and $g$ sends
$(b,c)$ to $b\otimes x - c\otimes y$. Moreover, if $\overline{B}$
is a hermitian vector bundle, then the monomorphism $f$ preserves
the hermitian metric. It follows that the hermitian vector bundle
$\tr\nolimits_1(\overline{E})$ is $\overline{E}^0(1)\oplus
\overline{E}^1$ and therefore
$$\ch(\tr\nolimits_1(\overline{E}))=\ch(\overline{E}^0(1)\oplus \overline{E}^1)=\ch(\overline{E}^0(1))
+ \ch(\overline{E}^1).$$ Since
$\ch(\overline{\mathcal{O}(1)})=1+\omega\in D_1^2+\mmW_1^2$, the
differential form $\ch(\overline{E}^0(1)) + \ch(\overline{E}^1)$
is zero in the complex $\bigoplus_{p\geq 0}\wmD_{\P}^{2p-*}(X,p)$.
\end{proof}

\begin{lmm} Let $n>0$ and let $\overline{E}\in \Z \widehat{C}_{n}(X\times (\P^1)^m)$ be a
hermitian $n$-cube which is split in the last direction, that is,
for every $\bj\in \{0,1,2\}^{n-1}$, the $1$-cube
$$(\partial_n^0 \overline{E})^{\bj} \rightarrow (\partial_n^1 \overline{E})^{\bj}  \rightarrow (\partial_n^2 \overline{E})^{\bj} $$
is hermitian split. Then, $$
\ch\nolimits_{n,m}(\overline{E})=0$$ in $\bigoplus_{p\geq 0}
\mmD^{2p-n-m}(X,p)$.
\end{lmm}
\begin{proof} Recall that if $\overline{E}$ is a hermitian $n$-cube
$$\tr\nolimits_n(\overline{E})=\tr\nolimits_1
\tr\nolimits_{n-1}(\overline{E})=\tr\nolimits_1\big(\tr\nolimits_{n-1}(\partial_n^0\overline{E})\rightarrow
\tr\nolimits_{n-1}(\partial_n^1\overline{E}) \rightarrow
\tr\nolimits_{n-1}(\partial_n^2\overline{E})\big).$$

Then, if $\overline{E}$ is split in the last direction, the 1-cube
$$ \tr\nolimits_{n-1}(\partial_n^0\overline{E})\rightarrow
\tr\nolimits_{n-1}(\partial_n^1\overline{E}) \rightarrow
\tr\nolimits_{n-1}(\partial_n^2\overline{E})$$ is orthogonally
split. Then, the result follows from Lemma \ref{chzero}.
\end{proof}

\begin{crl}\label{splitzero}
  Let $n>0$ and let $\overline{E}\in \Z \widehat{C}_{n}(X\times (\P^1)^m)$ be a
hermitian $n$-cube which is split in any direction. Then $\ch\nolimits_{n,m}(\overline{E})=0$ in the group $\bigoplus_{p\geq 0}
\mmD^{2p-n-m}(X,p)$.
  \end{crl}
  \vist

\begin{thm}\label{commut} Let $X$ be a proper arithmetic variety over $\Z$.
The diagram
$$  \xymatrix{
N \widehat{C}_*(X) \ar[r]^(0.37){\ch} \ar[d]_{\Psi^{k}}   & \bigoplus_{p\geq 0} \mmD^{2p-*}(X,p) \ar[d]^{\Psi^k} \\
\wZ \widehat{C}_*^{\widetilde{\P}}(X) \ar[r]_(0.37){\ch} & \bigoplus_{p\geq 0}
\mmD^{2p-*}(X,p) }$$
 is commutative.
\end{thm}
\begin{proof}
Let $\overline{E}$ be a hermitian $n$-cube.
If $n\geq 2$, then by Proposition \ref{adams2}, all the cubes in the image of $\Psi^k$ are hermitian split in at least one direction. Therefore, by
Corollary \ref{splitzero}, they vanish after applying ``$\ch$''. The same reasoning applies to some of the summands of the image of $\Psi^k$ when $n=1$. The rest of the terms are  hermitian $1$-cubes of the form $\mu(\overline{A})$ with $\overline{A}= \lambda_l(\Psi^{l}(\overline{F})^*)\otimes \overline{G}$ or $\overline{A}= \overline{F}\otimes \lambda_l(\Psi^l(\overline{G})^*)$ for some hermitian bundles $\overline{F},\overline{G}$. By Proposition \ref{koszulvanish}, these terms vanish as well after applying ``$\ch$''.

Finally, if $n=0$,  by Proposition \ref{gs} and the definition of $\Psi^k$ on
differential forms, we have
\begin{eqnarray*}
 \ch(\Psi^k(\overline{E}))
 &=& \frac{(-1)^n}{(2\pi i)^n}
\int_{(\P^1)^n}
\ch(\Psi^k(\tr\nolimits_n(\lambda(\overline{E}))))\wedge W_n
\\ & = & \frac{(-1)^n}{(2\pi i)^n} \int_{(\P^1)^n}
\Psi^k(\ch(\tr\nolimits_n(\lambda(\overline{E}))))\wedge W_n \\
&=& \Psi^k(\ch(\overline{E})).
\end{eqnarray*}
\end{proof}

\subsection{Adams operations on higher arithmetic $K$-theory} Let $X$ be a
proper arithmetic variety over $\Z$. Proposition \ref{cubes2} and
 Theorem \ref{commut} enable us to define, for every $k\geq 0$, the
\emph{Adams operation on higher arithmetic $K$-groups}:
\begin{enumerate*}[$\blacktriangleright$]
\item Since the simple complex associated to a morphism is a functorial construction, for every $k$ there is an Adams operation morphism on the
Deligne-Soul{\'e} higher arithmetic $K$-groups:
$$\Psi^k: \widehat{K}_n(X)_{\Q} \rightarrow  \widehat{K}_n(X)_{\Q},\quad n\geq 0.$$
\item By Proposition \ref{arithlong} for every $k$ there is an
Adams operation morphism on the Takeda higher arithmetic
$K$-groups:
$$\Psi^k: \widehat{K}^{T}_n(X)_{\Q} \rightarrow  \widehat{K}_n^{T}(X)_{\Q},\quad n\geq 0.$$
\end{enumerate*}

We have proved the following theorems.

\begin{thm}[(Adams operations)] Let $X$ be a proper arithmetic variety over $\Z$ and let
$\widehat{K}_n(X)$ be the $n$-th Deligne-Soul{\'e} arithmetic
$K$-group. There are Adams
operations
$$ \Psi^k: \widehat{K}_n(X)_{\Q}  \rightarrow \widehat{K}_n(X)_{\Q}, $$
compatible with the Adams operations in $K_n(X)_{\Q}$ and
$\bigoplus_{p\geq 0}H^{2p-n}_{\mmD}(X,\R(p))$, by means of the
morphisms $a$ and $\zeta$.
\end{thm}
\vist

\begin{thm}[(Adams operations)]\label{adamsariththeo1}
Let $X$ be a proper arithmetic variety over $\Z$ and let
$\widehat{K}^{T}_n(X)$ be the $n$-th arithmetic $K$-group defined
by Takeda in \cite{Takeda}. Then, for every $k\geq 0$ there exists
an Adams operation morphism
$\Psi^k:\widehat{K}_{n}^{T}(X)_{\Q}\rightarrow
\widehat{K}_{n}^{T}(X)_{\Q}$ such that the following diagram is
commutative:  {\small
$$ \xymatrix{ K_{n+1}(X)_{\Q} \ar[r]^(0.4){\ch} \ar[d]_{\Psi^k} &
\bigoplus_{p\geq 0} \widetilde{\mmD}^{2p-n-1}(X,p)
 \ar[r]^(0.65){a} \ar[d]_{\Psi^k} &
\widehat{K}_{n}^{T}(X)_{\Q} \ar[r]^{\zeta} \ar[d]_{\Psi^k} &  K_{n}(X)_{\Q} \ar[d]_{\Psi^k} \ar[r] & 0    \\
K_{n+1}(X)_{\Q} \ar[r]^(0.40){\ch}   & \bigoplus_{p\geq 0}
\widetilde{\mmD}^{2p-n-1}(X,p) \ar[r]^(0.65){a} &
\widehat{K}^{T}_n(X)_{\Q} \ar[r]^{\zeta} & K_n(X)_{\Q}  \ar[r] &
0. }
$$}  Moreover, the diagram
$$
\xymatrix{ \widehat{K}^{T}_n(X)_{\Q} \ar[r]^(0.4){\ch}
\ar[d]_{\Psi^k}& \bigoplus_{p\geq 0}  \mmD^{2p-n}(X,p)
\ar[d]^{\Psi^k}
\\ \widehat{K}^{T}_n(X)_{\Q} \ar[r]_(0.4){\ch} & \bigoplus_{p\geq 0} \mmD^{2p-n}(X,p) } $$
is commutative
\end{thm}
\vist

\medskip
\textbf{Lambda operations.}
Let $X$ be a proper arithmetic variety over $\Z$. Consider the
product structure on $\widehat{K}_{*}(X)_{\Q}$ defined before
 Lemma \ref{productarithktheory}. Then, by the relation between the Adams and $\lambda$ operations in a $\lambda$-ring (which is a $\Q$-algebra), there are induced $\lambda$-operations
$$\lambda^k: \widehat{K}_n(X)_{\Q}  \rightarrow \widehat{K}_n(X)_{\Q}.$$

\begin{crl}[(Pre-$\lambda$-ring)]\label{adamsariththeo2} Let $X$ be a proper arithmetic variety over $\Z$.
Then, $\widehat{K}_{*}(X)_{\Q}$ is a pre-$\lambda$-ring. Moreover,
there is a commutative square $$ \xymatrix{
\widehat{K}_{n}(X)_{\Q} \ar[r]^{\zeta} \ar[d]_{\lambda^k} &  K_{n}(X)_{\Q} \ar[d]^{\lambda^k} \\
\widehat{K}_n(X)_{\Q} \ar[r]^{\zeta} & K_n(X)_{\Q} }
$$
\end{crl}
\begin{proof}  The diagram is commutative since the Adams and
lambda operations in $K_*(X)$ are related under the product
structure on $K_*(X)$ which is zero in $\bigoplus_{n\geq
1}K_{n}(X)$.
\end{proof}

\begin{prp} Let $X$ be a proper arithmetic variety over $\Z$.
The Adams operations given here for $\widehat{K}_0(X)_{\Q}$ agree
with the ones given by Gillet and Soul{\'e} in
\cite{GilletSouleClassesII}.
\end{prp}
\begin{proof}
It follows from the definition.
\end{proof}

Consider  the product structure in $\bigoplus_{n\geq
0}\widehat{K}_n^{T}(X)_{\Q}$ having $\bigoplus_{n\geq
1}\widehat{K}_n^{T}(X)_{\Q}$ as a zero square ideal and agrees
with the product defined by Takeda in \cite{Takeda} otherwise.

\begin{crl}[(Pre-$\lambda$-ring)]\label{adamsariththeo3} Let $X$ be a proper arithmetic variety over $\Z$.
Then, $\widehat{K}_{*}^{T}(X)_{\Q}$ is a pre-$\lambda$-ring.
Moreover, there is a commutative square $$ \xymatrix{
\widehat{K}^{T}_{n}(X)_{\Q} \ar[r]^{\zeta} \ar[d]_{\lambda^k} &  K_{n}(X)_{\Q} \ar[d]^{\lambda^k} \\
\widehat{K}^{T}_n(X)_{\Q} \ar[r]^{\zeta} & K_n(X)_{\Q} }
$$
\end{crl}
\begin{proof} The proof is analogous to the proof of Corollary \ref{adamsariththeo2}.
\end{proof}

\begin{rem}
One way to prove that the pre-$\lambda$-ring structure on
$\widehat{K}_*(X)_{\Q}$ given here is actually a $\lambda$-ring structure, it
is necessary to find precise exact sequences relating, at the level of vector
bundles, the equalities in $K_0(X)$
\begin{eqnarray*}
\Psi^k(E\otimes F ) &=& \Psi^k(E)\otimes \Psi^k(F), \\
\Psi^k(\Psi^l(E)) &=& \Psi^{kl}(E).
\end{eqnarray*}
This implies finding formulas for $$\bigwedge\nolimits^k(E\otimes
F),\qquad
\bigwedge\nolimits^k\left(\bigwedge\nolimits^l(E)\right),
$$ in terms of tensor and exterior products. The theory of Schur functors, gives a formula
for the first term. However, the second formula is an open problem.
Nevertheless, even for the first equality, when we try to apply the formulas to
our concrete situation, the combinatorics become really complicated.

It would desirable and interesting to find a non-direct approach in order to prove these relations. One attempt could be
to go through the ``Arakelov'' representation ring on the linear group scheme over $\Z$, $\widehat{R}_{\Z}(GL_n\times GL_m)$ introduced by R\"ossler in \cite{Roessler1}.
\end{rem}

%

\providecommand{\bysame}{\leavevmode\hbox to3em{\hrulefill}\thinspace}
\providecommand{\MR}{\relax\ifhmode\unskip\space\fi MR }
\providecommand{\MRhref}[2]{%
  \href{http://www.ams.org/mathscinet-getitem?mr=#1}{#2}
}
\providecommand{\href}[2]{#2}

\end{document}